\documentclass[11pt,reqno]{amsart}
\usepackage{latexsym,amsmath}
\usepackage[left=3cm,top=2.5cm,right=3cm,bottom=2.5cm]{geometry}
\usepackage{color}
\usepackage{amsthm}
\usepackage{amssymb}
\usepackage{epsfig}
\usepackage{dsfont}

\def\PP{{\mathbb P}}
\def\EE{{\mathbb E}}

\let\epsilon\varepsilon

\def\({\left(}
\def\){\right)}  

\def\llfloor{\left\lfloor}
\def\rrfloor{\right\rfloor}
\def\llceil{\left\lceil}
\def\rrceil{\right\rceil}
\newcommand{\etal}{\textsl{et al.}}

\newcommand{\rbabs}[1]{\left|{#1}\right|}
\def \rb_ind{\mathds{1}}

\DeclareMathOperator{\homg}{hom}
\DeclareMathOperator{\sub}{sub}

\begin{document}

\newtheorem{theorem}{Theorem}[section]
\newtheorem{cor}[theorem]{Corollary}
\newtheorem{lemma}[theorem]{Lemma}
\newtheorem{fact}[theorem]{Fact}
\newtheorem{property}[theorem]{Property}
\newtheorem{proposition}[theorem]{Proposition}
\newtheorem{claim}[theorem]{Claim}
\newtheorem{definition}[theorem]{Definition}
\theoremstyle{definition}
\newtheorem{example}[theorem]{Example}
\newtheorem{remark}[theorem]{Remark}
\newcommand\eps{\varepsilon}
\newcommand\la {\lambda}
\newcommand{\E}{\mathbb E}
\newcommand{\Var}{{\rm Var}}
\newcommand{\Prob}{\mathbb{P}}
\newcommand{\N}{{\mathbb N}}
\newcommand{\eqn}[1]{(\ref{#1})}
\newcommand{\carlos}[1]{{\bf [~Carlos Mar 1:\ } {\em #1}{\bf~]}}

\def\llfloor{\left\lfloor}
\def\rrfloor{\right\rfloor}
\def\llceil{\left\lceil}
\def\rrceil{\right\rceil}

\title{Limits of permutation sequences through permutation regularity}

\author[C. Hoppen]{Carlos Hoppen}
\address{Instituto de Matem\'atica, UFRGS -- Avenida Bento Gon\c{c}alves, 9500,
 91509-900, Porto Alegre, RS, Brazil}
\email{\tt choppen@ime.usp.br}
\author[Y. Kohayakawa]{Yoshiharu Kohayakawa}
\address{Instituto de Matem\'atica e Estat\'\i stica, USP -- Rua do Mat\~ao
  1010, 05508--090 S\~ao Paulo, SP, Brazil} \email{\tt yoshi@ime.usp.br}
\author[C. G. Moreira]{Carlos Gustavo Moreira}
\address{IMPA -- Estrada Dona Castorina 110, 22460--320 Rio de
  Janeiro, RJ, Brazil} \email{\tt gugu@impa.br}
\author[R. M. Sampaio]{Rudini Menezes Sampaio}
\address{Departamento de Computa\c c\~ao, Centro de Ci\^encias, UFC --
  Campus do Pici, Bloco 910, 60451--760 Fortaleza, CE, Brazil} \email{\tt rudini@lia.ufc.br}

\thanks{The statements of some of the results of this paper have
  appeared in the Proc.~of the $21^{\textrm{st}}$ ACM-SIAM Symposium
  on Discrete Mathematics (SODA) (2010).  The results in
  Section~\ref{section_regularity} are stated in the Proc.~of the V
  Latin-American Algorithms, Graphs and Optimization Symposium (LAGOS)
  (2009)}

\thanks{The first author acknowledges the support by FAPERGS
  (Proc.~10/0388-2), FAPESP (Proc.~2007/56496-3), and CNPq
  (Proc.~484154/2010-9).  The second author was partially supported by
  CNPq (Proc.~308509/2007-2, 484154/2010-9). The fourth author was
  partially supported by Funcap (Proc.~07.013.00/09) and CNPq
  (Proc.~484154/2010-9).}

\begin{abstract}
A permutation sequence $(\sigma_n)_{n \in \mathbb{N}}$
is said to be convergent if, for every fixed permutation $\tau$, the density of occurrences of $\tau$ in the elements of the sequence converges. We prove that such a convergent sequence has a natural limit object, namely a Lebesgue measurable function $Z:[0,1]^2 \to [0,1]$ with the additional properties that, for every fixed $x \in [0,1]$, the restriction $Z(x,\cdot)$ is a cumulative distribution function and, for every $y \in [0,1]$, the restriction $Z(\cdot,y)$ satisfies a ``mass'' condition. This limit process is well-behaved: every function in the class of limit objects is a limit of some permutation sequence, and two of these functions are limits of the same sequence if and only if they are equal almost everywhere. An important ingredient in the proofs is a new model of random permutations, which generalizes previous models and is interesting for its own sake. 
\end{abstract}

\maketitle

\section{Introduction}

As usual, a \emph{permutation} of a finite set $S$ is a bijective function of $S$ into itself. We shall focus on permutations $\sigma$ on the set $S=\{1,\ldots,n\}=[n]$, where $n$ is a positive integer, called the \emph{length} of $\sigma$ and is denoted by $|\sigma|$. In this work, a permutation $\sigma$ on $[n]$ is represented by $\sigma=(\sigma(1),\ldots,\sigma(n))$, and the set of all permutations on $[n]$ is denoted by $S_n$. We denote by $\mathcal{S}=\bigcup_{i=1}^\infty S_n$  the set of all finite permutations. A \emph{graph} $G=(V,E)$ is given by its \emph{vertex set} $V$ and its \emph{edge set} $E\subseteq \{ \{u,v\} \subset V \colon u \neq v\}$. In particular, the graphs considered here do not contain loops or multiple edges.

The main goal of this paper is to introduce a notion of convergence of a permutation sequence $(\sigma_n)_{n \in \mathbb{N}}$ and to 
identify a natural limit object for such a convergent 
sequence whose associated sequence of lengths $(|\sigma_n|)_{n \in \mathbb{N}}$ tends 
to infinity. Lov\'{a}sz and Szegedy~\cite{lovasz06} were concerned with these questions 
in the case of graph sequences $(G_n)_{n \in \mathbb{N}}$. This has been further investigated by Borgs \etal~in~\cite{borgs06b} and~\cite{borgs06c}, where, among other things, limits of graph sequences were used to characterize the testability of graph parameters. The convergence of sequences of combinatorial objects has also been addressed in other structures.
 For instance, graphs with degrees bounded by a constant have been addressed in the recent works of
 Benjamini and Schramm~\cite{benjamini_schramm} and of Elek~\cite{elek1,elek}. 
 Elek and Szegedy~\cite{elek_szegedy} studied this problem for hypergraphs.
 See Lov\'{a}sz~\cite{lovasz} for a comprehensive survey of this area. 

Currently, the main application of our results in this paper is in property testing of permutations. Roughly speaking, the objective of testing is to decide whether a combinatorial structure satisfies some property, or to estimate the value of some numerical function associated with this combinatorial structure, by considering only a randomly chosen substructure of sufficiently large, but constant size. These problems are called \emph{property testing} and \emph{parameter testing}, respectively; a property or parameter is said to be \emph{testable} if it can be estimated accurately in this way. The algorithmic appeal of testability is evident, as, conditional on sampling, this leads to reliable constant-time randomized estimators for the said properties or parameters.
 In~\cite{rudini08d} the present authors address these questions through the prism of subpermutations. Among their main results are a permutation result in the direction of Alon and Shapira's~\cite{alon_shapira3} work on the testability of hereditary graph properties, and a permutation counterpart of the characterization of testable parameters by Borgs \etal~\cite{borgsstoc}.

Given the similarity of our results with the ones obtained in~\cite{lovasz06}, we briefly describe that work. 
Central in the arguments is the notion of a \emph{homomorphism} of a
 graph $F$ into a graph $G$, a function $\phi:V(F) \rightarrow V(G)$ that 
maps the vertex set $V(F)$ of $F$ into the vertex set $V(G)$ of $G$ with the property that, 
for every edge $\{u,v\}$ in $F$, the pair $\{\phi(u),\phi(v)\}$ is an edge in $G$.
 The number of homomorphisms of $F$ into $G$ is denoted by $\homg(F,G)$, while the 
\emph{homomorphism density} of $F$ into $G$ is given by the probability that a uniformly chosen $\phi$ is a homomorphism:
\begin{equation}
t(F,G)=\frac{\homg(F,G)}{|V(G)|^{|V(F)|}}.
\end{equation}

It is natural to measure the similarity between two graphs $G$ and $G'$ by comparing the homomorphism density of different graphs $F$ into them.
This suggests defining a graph sequence $(G_n)_{n \in \mathbb{N}}$ as being \emph{convergent} if, for every (simple) graph $F$, the sequence of real numbers $(t(F,G_n))_{n \in \mathbb{N}}$ converges. Lov\'{a}sz and Szegedy identify a natural limit object for such a convergent sequence in the form of a symmetric Lebesgue measurable function $W:[0,1]^2 \rightarrow [0,1]$, called a \emph{graphon}, which satisfies
$$\lim_{n\to\infty}t(F,G_n)\ =\ t(F,W)\ :=\ \int_{[0,1]^k}\prod_{ij\in E(F)}W(x_i,x_j)dx_1\cdots dx_k$$
for every (simple) graph $F$, where $k=|V(F)|$. Despite its somewhat daunting appearance, the integral on the right-hand side is just a natural generalization of homomorphism density to the context of functions. 

Moreover, they show that, given a graphon $W$, there exists a graph sequence $(G_n(W))_{n \in \mathbb{N}}$ converging to $W$. The proof of this result relies on a randomized construction interesting for its own sake: given a graphon $W:[0,1]^2 \rightarrow [0,1]$ and a positive integer $n$, a \emph{$W$-random} graph $G(n,W)$ with vertex set $[n]$ is generated as follows. First, $n$ real numbers $X_1, \ldots, X_n$ are generated independently according to the uniform probability distribution on the interval $[0,1]$. Now, for every pair of distinct vertices $i$ and $j$ in $[n]$, the pair $\{i,j\}$ is added to the edge set of the graph independently with probability $W(X_i,X_j)$. It is important to point out that this model of random graph
generalizes the random graph model $G(n,H)$ introduced by Lov\'{a}sz and S\'{o}s~\cite{lovasz_sos08}, further generalizing the classical Erd\H{o}s--R\'{e}nyi model $G_{n,p}$~\cite{erdos_renyi}. In the context of graph sequences, Lov\'{a}sz and Szegedy show that, for any fixed graphon $W$, the sequence $(G(n,W))_{n \in \mathbb{N}}$ converges to $W$ with probability one.  

In our paper, a similar path is traced for permutation sequences $(\sigma_n)_{n \in \mathbb{N}}$. However, several difficulties of a technical nature arise, as the limit objects obtained here are more constrained. The r\^{o}le of the homomorphism density $t(F,G_n)$ of a fixed graph $F$ into $G_n$ is played here by the subpermutation density $t(\tau,\sigma_n)$ of a fixed permutation $\tau$ into $\sigma_n$, which we now define. By $[n]^m_<$ we mean the set of $m$-tuples in $[n]$ whose elements are in strictly increasing order.
\begin{definition}[Subpermutation density]\label{def.dens.perm1}
For positive integers $k, n \in \N$, let $\tau \in S_k$ and $\pi \in S_n$. 
The \emph{number of occurrences} $\Lambda(\tau,\pi)$  of the permutation $\tau$ in $\pi$ is the number of $k$-tuples
 $(x_1,x_2,\ldots,x_k) \in [n]^k_<$ such that
 $\pi(x_i) < \pi(x_j)$ if and only if $\tau(i)<\tau(j)$. 
The \emph{density} of the permutation $\tau$ as a \emph{subpermutation} of $\pi$ is given by 
\begin{equation}\label{def_subperm_density_formula_cases}
t(\tau,\pi)=
\begin{cases}
\binom{n}{k}^{-1}\Lambda(\tau,\sigma)  & \textrm{if } k\leq n\\
\;0  & \textrm{if  } k> n.
\end{cases}
\end{equation}
\end{definition}
As an illustration, there is a subpermutation  $\tau=(3,1,4,2)$ in $\sigma=(5,6,2,4,7,1,3)$, since $\sigma$ maps the index set $(1,3,5,7)$ onto $(5,2,7,3)$, which appears in the relative order given by~$\tau$. This concept may be used to define a convergent permutation sequence in a natural way: it is a sequence for which the densities of subpermutations of any given type converge. Formally, we have the following.
\begin{definition}[Convergence of a permutation sequence]\label{def_conv}
A permutation sequence $(\sigma_n)_{n \in \mathbb{N}}$ is \emph{convergent} if, for every fixed permutation $\tau$, the sequence of real numbers $(t(\tau,\sigma_n))_{n \in \mathbb{N}}$ converges.
\end{definition}
The interesting case occurs when the sequence of lengths $(|\sigma_n|)_{n \in \mathbb{N}}$ tends to infinity, since, as we shall see, every convergent permutation sequence $(\sigma_n)_{n \in \mathbb{N}}$ is otherwise eventually constant.

We prove that, when this is the case, any convergent permutation sequence has a natural limit object, called a \emph{limit permutation}. This limit object consists of a family of \emph{cumulative distribution functions}, or \emph{cdf} for short, which, following Lo\`{e}ve (see~\cite{loeve77}, Chapter III, \S 10), are non-decreasing left-continuous functions $F:[0,1]\to[0,1]$ with $F(0)=0$.

\begin{definition}[Limit permutation]\label{def.perm.lim}
A \emph{limit permutation} is a Lebesgue measurable function $Z:[0,1]^2 \to[0,1]$ satisfying the following conditions:
\begin{itemize}
\item[(a)] for every $x\in[0,1]$, the function $Z(x,\cdot)$ is a cdf, continuous at the point $0$, for which $Z(x,1)=1$;

\item[(b)] for every $y\in[0,1]$, the function $Z(\cdot,y)$ is such that
$$\int_0^1 Z(x,y)\ dx\ =\ y.$$
\end{itemize}
\end{definition}
We observe that the continuity requirements in part (a) are not crucial. As a matter of fact, one could define limit permutations even if the continuity of $Z(x,\cdot)$ at zero and the condition $Z(x,1)=1$ were dropped, since the remaining conditions imply that the set of points $x \in [0,1]$ for which at least one of these two requirements fails to hold must have Lebesgue measure zero. Nevertheless, a few technical aspects can be avoided when these conditions are satisfied, and we therefore assume that they hold.

As with graphs, limit permutations may be used to define a model of random permutations. Here, a sequence of $n$ real numbers $X_1< \cdots< X_n$ is generated uniformly in the simplex $[0,1]_{<}^n$. We then choose a second sequence of random variables $a_1,\ldots,a_n$, independently, with probabilities induced by the cdfs $Z(X_1,\cdot),\ldots,Z(X_n,\cdot)$ associated with the limit permutation $Z$, respectively. The \emph{$Z$-random permutation} $\sigma(n,Z)$ is given by the indices of the real numbers $a_i$ as these are listed in increasing order. For example, if $n=3$ and the generation of the $a_i$ yields $a_2 < a_1 <a_3$, then $\sigma(3,Z)=(2,1,3)$. We shall see that, with probability one, $a_i\not= a_j$ for every $i,j\in[n]$. This new model of random permutation generalizes the classical random permutation model, in which a permutation is selected uniformly at random from all permutation of $n$. Indeed, a classical random permutation may be obtained as a $Z$-random permutation for the limit permutation $Z_u$, where $Z_u(x,y)=y$ for all $(x,y)\in[0,1]^2$.

Again inspired by the graph case, given a limit permutation $Z$, we may define the subpermutation density $t(\tau,Z)$ of a permutation $\tau$ on $[k]$ in $Z$ as the probability that the $Z$-random permutation $\sigma(k,Z)$ is equal to $\tau$. We shall give an alternative, more technical definition of the subpermutation density $t(\tau,Z)$ in Section~\ref{section_convergence} (see Definition~\ref{def.dens.subperm.lim}). We may now state our main result.
\begin{theorem}[Main result]\label{teorema_principal}
Given a convergent permutation sequence $(\sigma_n)_{n \in \mathbb{N}}$ for which $\lim_{n \to \infty}|\sigma_n|=\infty$, there exists a limit permutation $Z:[0,1]^2\to[0,1]$ such that
$$\lim_{n\to\infty}t(\tau,\sigma_n) = t(\tau,Z) \textrm{ for every permutation } \tau.$$
Conversely, every limit permutation $Z:[0,1]^2\to[0,1]$ is a limit of a convergent permutation sequence.
\end{theorem}

Another important question regarding limits of permutation sequences is a characterization of the set of limit permutations that are limits to the same given convergent permutation sequence $(\sigma_n)_{n \in \mathbb{N}}$. It is clear that this limit is not unique in a strict sense, because, if $Z$ is the limit of a given permutation sequence $(\sigma_n)_{n \in \mathbb{N}}$ and $A$ is a measurable subset of $[0,1]$ with measure zero, then any limit permutation obtained through the replacement of the cdf $Z(x,\cdot)$ by a cdf $Z^{\ast}(x,\cdot)$ for every $x \in A$ is also a limit of $(\sigma_n)_{n \in \mathbb{N}}$, as the value of the integrals will not be affected. A similar situation occurred in the case of graph limits in~\cite{borgs06b}, and uniqueness was captured by an equivalence relation induced by a pseudometric $d_{\square}$ between graphons. Two graphons $W$ and $W^{\ast}$ were proved to be limits of the same graph sequence if and only if $d_{\square}(W,W^{\ast})=0$. We use a similar proof technique for permutations, but more can be achieved: two limit permutations $Z_1$ and $Z_2$ are limits of the same sequence if and only if they are equal almost everywhere, that is, if the set $\{(x,y):Z_1(x,y) \neq Z_2(x,y)\}$
has measure zero.
\begin{theorem}\label{res_princ1} Let $Z_1, Z_2$ be limit permutations.
Let $(\sigma_n)_{n \in \mathbb{N}}$ be a convergent permutation sequence.
Then  $\sigma_n \to Z_1$ and $\sigma_n \to Z_2$  if and only if the set $\{x : Z_1(x,\cdot) \not \equiv Z_2(x,\cdot)\}$
has Lebesgue measure zero.
\end{theorem}

On the other hand,  based on a previous concept by Cooper~\cite{cooper1},
 we may introduce a distance $d_{\square}$ between permutations (see \eqref{def.rec}) and,
 more generally, between limit permutations, which is a permutation counterpart of the graph pseudometric discussed in the previous paragraph. In particular, we may characterize our notion of convergence of permutation sequences in terms of this metric. As usual, a sequence $(\sigma_n)_{n \in \mathbb{N}}$ is 
said to be a \emph{Cauchy sequence} with respect to the metric $d_{\square}$ if, for every $\eps>0$, 
there exits $n_0=n_0(\eps)$ such that $d_{\square}(\sigma_n,\sigma_m)<\eps$ for every $n,m \geq n_0$.

\begin{theorem}\label{teorema_equiv}
A permutation sequence $(\sigma_n)_{n \in \mathbb{N}}$ converges if and only if it is a Cauchy sequence with respect to the metric $d_{\square}$.
\end{theorem}

Also in analogy with the work for graphs by Lov\'{a}sz and Szegedy, the theory
 in this paper can be considered in terms of the discrete metric space $(\mathcal{S},d_{\square})$,
 where $\mathcal{S}=\bigcup_{i=1}^\infty S_n$ is the set of all finite permutations and $d_{\square}$ is
 the metric of the previous paragraph. By a standard diagonalization argument, every permutation 
sequence can be shown to have a convergent subsequence (see Lemma~\ref{conv_subseq}). 
As a consequence, the metric space $(\mathcal{S},d_{\square})$ can be enlarged to a compact metric space $(\mathcal{Z}/_{\sim},d_{\square})$ by adding limit permutations, where we identify limit permutations that are equal almost everywhere. By Theorem~\ref{teorema_principal} the subspace of permutations is dense in $\mathcal{Z}/_{\sim}$; moreover, it is discrete, since a sequence cannot converge to a permutation without being eventually constant (Claim \ref{claim_eventually_constant}).
 Finally, Theorem~\ref{teorema_equiv} tells us that, when restricted to permutations, convergence
 in this metric space coincides with the concept of convergence in Definition~\ref{def_conv}.

 We have found two essentially different paths for establishing the
 main results in this paper. The approach followed here starkly
 resembles the work of Lov\'{a}sz and Szegedy~\cite{lovasz06} for
 graph sequences, which relies on Szemer\'{e}di-type regularity
 arguments. However, several difficulties of technical nature arise,
 as the limit objects here are more constrained than in the graph
 case. An alternative approach, taken in~\cite{balazs}, has a
 distinctive probabilistic flavor and has the advantage of being both
 more compact and more direct, since several of the technicalities of
 the current approach can be avoided.

This paper is structured as follows. Section~\ref{section_regularity} contains preliminary results and definitions. For instance, we introduce a distance $d_{\square}$ between permutations, based on a previous notion by Cooper~\cite{cooper1}, which is an important tool in our proofs. We also prove some basic facts about the convergence of permutation sequences. In Section~\ref{section_convergence}, we provide the proof that every convergent permutation sequence has a limit of the prescribed form. However, several technical aspects of this proof, which require some basic measure-theoretical results, are postponed to Section~\ref{section_technical}. Section~\ref{sec_random} deals with $Z$-random permutations, which allow us to demonstrate that any limit permutation is the limit of a convergent permutation sequence. Finally, Section~\ref{sec_uniqueness} is devoted to the proof of uniqueness, in the form stated in Theorem~\ref{res_princ1}.  

\section{Preliminaries}\label{section_regularity}

The present section deals with the concept of rectangular distance for permutations, which plays an important r\^{o}le in our proofs. Moreover, we establish some basic facts about the convergence of permutation sequences.

A fundamental tool in the work of Lov\'{a}sz and Szegedy~\cite{lovasz06} is a weaker version, due to Frieze and Kannan \cite{frieze_kannan}, of the powerful Regularity Lemma introduced by Szemer\'{e}di~\cite{szem76}. A similar framework can be developped for permutations. Following the work of Cooper~\cite{cooper2}, we first encode permutations as graphs.
\begin{definition}[Graph of a permutation]
Given a permutation $\sigma:[n] \rightarrow [n]$, the \emph{graph} $G_{\sigma}$ of $\sigma$ is the bipartite graph  with disjoint copies of $[n]$ as color classes $A$ and $B$, where $\{a,b\}$, with $a \in A$ and $b \in B$, is an edge if and only if $\sigma(a)<b$.  
\end{definition}

Let $I[n]$ be the set of all \emph{intervals} in $[n]$, that is, the set of all subsets of the form $\{x \in [n]~:~ a \leq x < b\}$, where $a,b \in [n+1]$ are called the \emph{endpoints} of the interval. Given a permutation $\sigma:[n] \to [n]$, Cooper~\cite{cooper2} defines the \emph{discrepancy of $\sigma$} as
$$D(\sigma)=\max_{S,T\in I[n]}  \left||\sigma(S)\cap T|-\frac{|S||T|}{n}\right|.$$
This is used to measure the ``randomness'' of a permutation. Indeed, sequences with low discrepancy, i.e., for which $D(\sigma)=o(n)$, are said to be \emph{quasi-random}. We use a normalized version of the same concept to introduce a distance between permutations.
\begin{definition}[Rectangular distance]\label{def.rec}
Given permutations $\sigma_1,\sigma_2:[n] \rightarrow [n]$, the \emph{rectangular distance} between $\sigma_1$ and $\sigma_2$ is given by
\[
  d_{\square}(\sigma_1,\sigma_2)\ =\ \frac{1}{n}\max_{S,T\in I[n]}\ \Big||\sigma_1(S)\cap T|-|\sigma_2(S)\cap T|\Big|.
\]
\end{definition}

As mentioned in the introduction, we wish to relate a convergent permutation sequence with a ``limit object", namely a bivariate function with range $[0,1]^2$ satisfying some special properties. This puts in evidence the need of relating the latter with a sequence of finite, dicrete objects. A first step in this transition is the concept of \emph{weighted permutation}, in the spirit of the weighted graphs introduced in~\cite{lovasz06}.

\begin{definition}[Weighted permutation]
Given an integer $k$, a \emph{weighted permutation} is a matrix $Q:[k]^2 \to [0,1]$ with the following two properties:
\begin{itemize}

\item[(a)] for every $i \in [k]$ and $j \leq j' \in [k]$, we have $Q(i,j) \leq Q(i,j')$;

\item[(b)] for every $j \in [k]$, we have
$$j-1 \leq \sum_{i=1}^k Q(i,j) \leq j.$$
\end{itemize}
\end{definition}

An important class of weighted permutations is given as follows. It is based on equitable partitions of a permutation, which we now define.

\begin{definition}
Let $n>0$ be an integer. A \emph{$k$-partition $P=(C_i)_{i=1}^k$ of $[n]$} is a partition of $[n]$ into $k$ consecutive intervals $C_1,\ldots,C_k$. A partition $P$ is said to be \emph{equitable} if $\displaystyle{\left||C_i|-|C_j|\right| \leq 1}$ for every $i,j \in [k]$.
\end{definition}

Let $\sigma:[n] \to [n]$ be a permutation and let $P=(C_i)_{i=1}^k$ be an equitable $k$-partition of $[n]$. The \emph{partition matrix} of $\sigma$ induced by $P$ is the matrix $Q_{\sigma,P}:[k]^2 \rightarrow [0,1]$ for which, given $u,w \in [k]$,
$$Q_{\sigma,P}(u,w)=\frac{e_{\sigma}(C_u,C_w)}{|C_u||C_w|},$$
where $e_{\sigma}(C_u,C_w)$ is the number of edges $\{a,b\}$ in the graph $G_\sigma$ with $a \in C_u \subset A$ and $b \in C_w \subset B$. If each interval in $P$ contains a single element, the partition matrix $Q_{\sigma,P}$ is the bipartite adjacency matrix of $G_{\sigma}$, which is denoted by $Q_\sigma$ and is itself a weighted permutation.

The following result relates partition matrices and weighted permutations. As with the other results in this section, the proof is easy and is omitted. It has originally appeared in~\cite{rudini08a}, and a complete proof may be found in~\cite{rudini08b}. For convenience, the proofs are also given in Appendix~\ref{apend_A}.
\begin{lemma}\label{prop_weighted}
If $\sigma:[n] \to [n]$ is a permutation and $P=(C_i)_{i=1}^k$ is an equitable $k$-partition of $[n]$, where $n>4k^2$, then the partition matrix $Q_{\sigma,P}$ is a weighted permutation.
\end{lemma}

We wish to extend the definition of rectangular distance between permutations to a distance between weighted permutations. Note that, for permutations $\sigma_1$ and $\sigma_2$ on $[n]$, and intervals $S$ and $T$ of $[n]$, where $T=\{x \in [n]~:~a \leq x <b\}$, we have
\begin{equation*}
\begin{split}
|\sigma_1(S) \cap T| - |\sigma_2(S) \cap T|&= \sum_{x \in S} \left( |\sigma_1(x) \cap T| - |\sigma_2(x) \cap T| \right) \\
&=  \sum_{x \in S} \Big(\left(Q_{\sigma_1}(x,b) - Q_{\sigma_1}(x,a) \right)-\left( Q_{\sigma_2}(x,b) - Q_{\sigma_2}(x,a)\right) \Big).
\end{split}
\end{equation*}
Therefore, the following extension is natural. To simplify notation, for a weighted permutation $Q:[k]^2\to [0,1]$, we henceforward assume that $Q(i,0)=0$ and $Q(i,k+1)=1$ for every $i \in [k]$.
\begin{definition}\label{def.dist.pond}
Given weighted permutations $Q_1,Q_2:[k]^2\to[0,1]$,
the \emph{rectangular distance} between $Q_1$ and $Q_2$ is given by
\[
  d_{\square}(Q_1,Q_2)\ =\ \frac{1}{k}\max_{\substack{S\in I[k]\\a<b\in[k+1]}}\
  \Big|\sum_{x\in S}\Big((Q_1(x,b)-Q_1(x,a))-(Q_2(x,b)-Q_2(x,a))\Big)\Big|.
\]
\end{definition}

For $n>0$ and an equitable $k$-partition $P=(C_i)_{i=1}^k$ of $[n]$, and given a matrix $Q:[k]^2\to[0,1]$, we let the \emph{blow-up matrix} $\mathcal{K}(P,Q):[n]^2\to[0,1]$ be the matrix obtained by replacing a single entry $(i,j)$ of $Q$ by a block of size $|C_i||C_j|$ assuming the same value. In other words,
$$\mathcal{K}(x,y)=Q(i,j) \textrm{ for every }x,y\in[n] \textrm{ with }x\in C_i \textrm{ and } y\in C_j.$$
With this definition, we may derive a partitioning $P$ of a permutation that, in some sense, resembles the weakly regular partitioning introduced for graphs by Frieze and Kannan~\cite{frieze_kannan}. The key property of such a partition $P$ is that, with respect to the rectangular distance, the bipartite adjancency matrix of the graph of the permutation can be well-approximated by the blow-up matrix of the partition matrix with respect to $P$. We point out, however, that unlike in the graph case and due to our restriction to considering subintervals of $[n]$ as opposed to more general subsets, this partitioning does not convey any particular structural information of the partitioned permutation. In spite of this, it still provides us with a useful encoding of permutations.
\begin{lemma}\label{prop_weak}
Given $\eps>0$, there exists $k_0>0$ such that, for every $k>k_0$ and every $n>2k$, the following property holds. If $P$ is an equitable $k$-partition of $[n]$ and $\sigma:[n] \to [n]$ is a permutation, then  
\[
  d_{\square}(Q_{\sigma},\mathcal{K}(P,Q_{\sigma,P}))\leq\varepsilon.
\]
\end{lemma}

We say that one such partition $P$ is a \emph{weak $\eps$-regular partition} of $\sigma$.

The concept of subpermutation density in a permutation can also be easily extended to weighted permutations.
\begin{definition}[Subpermutation density in a weighted permutation]
\label{def.dens.perm2}
Given a weighted permutation $Q:[k]^2\to[0,1]$
and a permutation $\tau:[m]\to[m]$, $m<k$, the \emph{subpermutation density} of $\tau$ in $Q$ is given by
\[
  t(\tau,Q)=\binom{k}{m}^{-1}\sum_{X\in[k]^m_<}\sum_{A\in[k+1]^m_<}
  \prod_{i=1}^m\Big(Q(x_i,a_{\tau(i)})-Q(x_i,a_{\tau(i)}-1)\Big),
\]
where we use the notation $X=(x_1,\ldots,x_m)$ and $A=(a_1,\ldots,a_m)$.
\end{definition}

It is not hard to see that, for $n>m$ and a permutation $\sigma:[n]\to[n]$, we have $t(\tau,\sigma)=t(\tau,Q_{\sigma})$.
The next result evinces the relationship between subpermutation density and rectangular distance.

\begin{lemma}\label{lema.cut.subpermut}
Let $\tau$ be a permutation and let $n$ be a positive integer with $n \geq 2|\tau|$. Then, given weighted permutations $Q_1,Q_2:[n]^2 \rightarrow [0,1]$, we have
$$|t(\tau,Q_1)-t(\tau,Q_2)| \leq 2|\tau|^2\cdot d_{\square}(Q_1,Q_2).$$
In particular, if $\sigma_1$ and $\sigma_2$ are permutations on $[n]$, we have
$|t(\tau,\sigma_1)-t(\tau,\sigma_2)| \leq 2|\tau|^2 d_{\square}(\sigma_1,\sigma_2)$.
\end{lemma}

To conclude this section,  we prove two simple facts about the convergence of permutation sequences that have been mentioned in the introduction. We show that every convergent sequence $(\sigma_n)_{n \in \N}$ such that $|\sigma_n| \not \to \infty$ must be eventually constant, and we establish that every permutation sequence contains a convergent subsequence.

For the first, we remind the reader that Theorem~\ref{teorema_principal} is only stated for permutation sequences $(\sigma_n)_{n \in \N}$ whose lengths tend to infinity, as we claimed that every other convergent sequence is eventually constant. In light of this, we prove this claim prior to addressing the main results. As before, if $\sigma \in S_n$ we write $\rbabs{\sigma}=n$. Recall the notion of convergence of permutation sequences from Definition \ref{def_conv}. 
\begin{claim}\label{claim_eventually_constant}
 Let $(\sigma_n)_{n \in \mathbb{N}}$ be a convergent permutation sequence such that $|\sigma_n| \not \to \infty$. Then
 the  sequence $(\sigma_n)_{n \in \mathbb{N}}$ is eventually constant, that is,
there is a permutation $\sigma$ and an $n_0 \in \N$ such that $n \geq n_0$ implies $\sigma_n = \sigma$.
\end{claim}
\begin{proof}
It follows from \eqref{def_subperm_density_formula_cases} that we have 
$
\sum_{\tau \in S_k} t(\tau, \pi)=\rb_ind [ k \leq \rbabs{\pi} ]
$ for
any $k \in \N$ and any permutation $\pi$, where $\rb_ind [A]$ denotes the indicator random variable for the event $A$. By Definition \ref{def_conv} we get that for any fixed $k \in \N$ the limit
$\lim_{n \to \infty}  \sum_{\tau \in S_k} t(\tau, \sigma_n)=\lim_{n \to \infty} \rb_ind [ k \leq \rbabs{\sigma_n} ] $
 exists and must be equal to $0$ or $1$.

From this and our assumption that $\liminf_{n \to \infty} \rbabs{\sigma_n} <+\infty$ it is easy to deduce that
$\liminf_{n \to \infty} \rbabs{\sigma_n}=\limsup_{n \to \infty} \rbabs{\sigma_n}$, thus there is some $m \in \N$ such that
$\rbabs{\sigma_n}=m$ if $n$ is large enough. 

Now if $\tau, \pi \in S_m$ then $t(\tau,\pi)=\rb_ind[\tau=\pi]$, thus $\lim_{n \to \infty} t(\tau, \sigma_n)$ must be equal to $0$ or $1$ for
all $\tau \in S_m$.
From this and $\rbabs{S_m}=m!<+\infty$ it is straightforward to deduce that the  sequence $(\sigma_n)_{n \in \mathbb{N}}$ is eventually constant.
\end{proof}

To conclude this section, we show that permutation sequences always contain convergent subsequences.
\begin{lemma}\label{conv_subseq}
Every permutation sequence has a convergent subsequence.
\end{lemma}

\begin{proof}
Let $(\sigma_n)_{n \in \mathbb{N}}$ be a permutation sequence.
 We shall find a convergent subsequence of $(\sigma_n)_{n \in \mathbb{N}}$ by a standard diagonalization argument.
 Since $S_n$ is finite for every $n \geq 1$, the set of $S=\bigcup_{n=1}^{\infty} S_n$ of all finite permutations is countable,
 say $S=(\tau_m)_{m \in \mathbb{N}}$.

If $(\sigma_n)_{n \in \mathbb{N}}$ does not converge, starting with $\tau_1$,
 we let $(\sigma_n^1)_{n \in \mathbb{N}}$ be a subsequence of $(\sigma_n)_{n \in \mathbb{N}}$ for which the bounded real
 sequence $(t(\tau_1,\sigma^1_n))_{n \in \mathbb{N}}$ converges. Inductively, for $m \geq 2$, we let $(\sigma_n^m)_{n \in \mathbb{N}}$ be 
a subsequence of $(\sigma_n^{m-1})_{n \in \mathbb{N}}$ such that $(t(\tau_m,\sigma^{m-1}_n))_{n \in \mathbb{N}}$ converges. 
It is now easy to see that the diagonal sequence $(\sigma_n^n)_{n \in \mathbb{N}}$ is such that, for every positive integer $m$, 
the sequence $(t(\tau_m,\sigma_n^n))_{n \in \mathbb{N}}$ converges. In other words, the sequence $(\sigma_n^n)_{n \in \mathbb{N}}$ is 
a convergent subsequence of $(\sigma_n)$.
\end{proof}

\section{Convergence of a permutation sequence}\label{section_convergence}

This section is devoted to proving the following theorem.
\begin{theorem}\label{teorema_principal1}
Let $(\sigma_n)_{n \in \mathbb{N}}$ be a convergent permutation sequence satisfying the condition $\lim_{n \to \infty} |\sigma_n|=\infty$. Then there exists a limit permutation $Z:[0,1]^2\to[0,1]$ such that
$$\lim_{n\to\infty}t(\tau,\sigma_n) = t(\tau,Z) \textrm{ for every permutation } \tau.$$
\end{theorem}
This is precisely the part of Theorem~\ref{teorema_principal} concerned with the existence of a limit for a convergent permutation sequence.

The statement of Theorem~\ref{teorema_principal1} still depends on the definition of $t(\tau,Z)$, the subpermutation density of $\tau$ in $Z$, which we have mentioned in the introduction in connection with random permutations. We shall now give an alternative definition of this concept. Some introductory discussion is needed: it is known that, associated with every cdf $F:[0,1]\to[0,1]$, there is a Lebesgue-Stieltjes probability measure $\mu_F$ over the Borel sets of $[0,1]$, namely the measure satisfying $\mu_F([0,1])=1$ and  $\mu_F([a,b))=F(b)-F(a)$, for $0\leq a<b \leq 1$. Moreover, given $k$ cdfs $F_1,\ldots,F_k$ and their respective measures $\mu_1,\ldots,\mu_k$, the product measure $\mu=\mu_1\times\cdots\times\mu_k$ is a probability measure in $[0,1]^k$ over the Borel sets of $[0,1]^k$. These facts are proved in~\cite{loeve77}, Chapter III, \S 10.

The usual notation for the Lebesgue-Stieltjes integral of a Borel-measurable function $g:[0,1]\to \mathbb{R}$ over the measure $\mu_1$ is $\int_{[0,1]}g\ d\mu_1\ =\ \int_{[0,1]}g\ dF_1$. When $S=[a,b]$, we use $\int_{[a,b]} g\ dF_1 = \int_a^b  g\ dF_1$. For a Borel-measurable function $g:[0,1]^k\to \mathbb{R}$, the integral with respect to $\mu=\mu_1\times\cdots\times\mu_k$ is denoted by
$$\displaystyle{\int_{[0,1]^k}g\ d\mu\ =\ \int_{[0,1]^k}g\ dF_1\cdots dF_k}.$$

\begin{definition}\label{def.LtauZ}
For $m>1$, let $\tau:[m]\to[m]$ be a permutation and let $Z:[0,1]^2\to[0,1]$
be such that $Z(x,\cdot)$ is a cdf for every $x\in[0,1]$. The \emph{core function} of $Z$ is the function  $L_{\tau,Z}:[0,1]^m\to[0,1]$ mapping $x=(x_1,\ldots,x_m)\in[0,1]^m$ to
\[
  L_{\tau,Z}(x)\ =\ \int_{[0,1]^m_<}dZ(x_{\tau^{-1}(1)},\cdot) \ dZ(x_{\tau^{-1}(2)},\cdot)\cdots
  \ dZ(x_{\tau^{-1}(m)},\cdot).
\]
\end{definition}
Observe that integration is taken over the $m$-simplex $[0,1]^m_< \subset [0,1]^m$, which is the set of $m$-tuples $(y_1,\ldots,y_m)$ such that $0\leq y_1<y_2<\cdots<y_m \leq 1$. The order of the integrating factors $dZ(x_{\tau^{-1}(1)},\cdot) \ dZ(x_{\tau^{-1}(2)},\cdot) \cdots \ dZ(x_{\tau^{-1}(m)},\cdot)$ in the product measure reflects the connection between the integration variable corresponding to $y_i$ and the measure associated with the cdf $Z(x_{\tau^{-1}(i)},\cdot)$. Also note that, in the above definition, the function $Z$ need not satisfy the mass condition in the definition of limit permutations.

We are now ready to define $t(\tau,Z)$ for a limit permutation $Z$.
In this definition, it is necessary to integrate the core function $L_{\tau,Z}(\cdot)$, whose measurability is established in Proposition~\ref{propQtoZ}.
\begin{definition}[Subpermutation density in a limit permutation]\label{def.dens.subperm.lim}
For a permutation $\tau:[m]\to[m]$ ($m>1$) and a limit permutation $Z:[0,1]^2\to[0,1]$,
the \emph{subpermutation density of $\tau$ in $Z$} is given by
\[
  t(\tau,Z)\ =\ m!\int_{[0,1]^m_<}L_{\tau,Z}(x)dx\ 
\]
\[
  =\ m!\int_{[0,1]^m_<}\Big(\int_{[0,1]^m_<}dZ(x_{\tau^{-1}(1)},\cdot) \ \cdots \ dZ(x_{\tau^{-1}(m)},\cdot)\Big)
  \ dx_1\cdots dx_m.
\]
\end{definition}
The factor $m!$ acts as a normalizer, since $\lambda([0,1]^m_<)=1/m!$, where $\lambda$ denotes the usual Lebesgue measure.

We may now concentrate on the proof of Theorem~\ref{teorema_principal1}.
This depends on the construction of a limit object $Z$, which we now sketch.
As in Section 5.4 of~\cite{lovasz06}, we first construct a subsequence $(\sigma_m')$ of $(\sigma_n)$ consisting of permutations with well-behaved weakly regular partitions, for which the partition matrices approximate a structured sequence $(Q_j)$ of weighted permutations.

With a weighted permutation $Q:[k]^2\to[0,1]$, we may associate a \emph{step function} $Z_Q:[0,1]^2 \to [0,1]$
which takes $(x,y)\in[0,1]^2$ to
\begin{equation}\label{eqdefZ}
\begin{matrix}
\\
Z_{Q}(x,y)\\
\\
\end{matrix}
=
\left\{
\begin{matrix}
0, &\textrm{ if }y=0;\\
\lceil ky \rceil/k, &\textrm{ if }x=0;\\
Q(\lceil kx \rceil,\lceil ky \rceil), &\textrm{ otherwise.}
\end{matrix}
\right.
\end{equation}
For example, in the case when $Q=Q_\sigma$ for some permutation $\sigma$ on $[n]$, we have
\begin{equation}\label{eqdefZsigma}
\begin{matrix}
\\
Z_\sigma(x,y) = Z_{Q_\sigma}(x,y)\\
\\
\end{matrix}
=
\left\{
\begin{matrix}
0, &\textrm{ if }y=0;\\
\lceil ny \rceil/n, &\textrm{ if }x=0;\\
Q_\sigma(\lceil nx \rceil,\lceil ny \rceil), &\textrm{ otherwise.}
\end{matrix}
\right.
\end{equation}
Recall that $Q_\sigma(\lceil nx \rceil,\lceil ny \rceil)=1$ if $\sigma(\lceil nx \rceil)<\lceil ny \rceil$ and $Q_\sigma(\lceil nx \rceil,\lceil ny \rceil)=0$ otherwise.

Therefore, with the sequence of weighted permutations $(Q_j)_{j \in \mathbb{N}}$, we may associate a sequence $(Z_j)_{j \in \mathbb{N}}$ of step functions $Z_j:[0,1]^2\to[0,1]$. We first show that the subpermutation density of some permutation $\tau$ in a weighted permutation $Q$ is well-approximated by the subpermutation density of $\tau$ in the corresponding step function $Z_Q$.
\begin{lemma}\label{lema_auxiliar}
Let $Q$ be a weighted permutation of order $n$ and let $Z_Q$ be the
step function associated with $Q$. Then, if $\tau$ is a permutation of length $m<n$, we have
$$\Big(1-\frac{m}{n}\Big)^mt(\tau,Q)\ \leq\ t(\tau,Z_Q)\ \leq\
t(\tau,Q)\ +\ \frac{(m+2)!}{n}$$
Consequently, for every $\eps>0$, $|t(\tau,Z_Q)-t(\tau,Q)|<\eps$ if $n$ is sufficiently large.
\end{lemma}

\begin{proof}
Let $Z=Z_Q$.  Let $[0,1]^m_{<,good}\subseteq[0,1]^m_<$ be the subset
of $m$-tuples $(x_1<x_2<\ldots<x_m)$ such that $\lfloor x_i\cdot
n\rfloor\not=\lfloor x_{i+1}\cdot n\rfloor$ for every $1\leq i<m$.

Since $t(\tau,Z)\ =\ m!\int_{[0,1]^m_<}L_{\tau,Z}(x)dx$, then
\[
  m!\int_{[0,1]^m_{<,good}}L_{\tau,Z}(x)dx\ \leq\ t(\tau,Z)\ \leq\
m!\int_{[0,1]^m_{<,good}}L_{\tau,Z}(x)dx\ +\
m!\cdot\lambda([0,1]^m_<\backslash[0,1]^m_{<,good}),
\]
where $\lambda$ is the Lebesgue measure in $[0,1]^m$.

It is not hard to see that
\[
 \int_{[0,1]^m_{<,good}}L_{\tau,Z}(x)dx\ =\ \frac{1}{n^m}\binom{n}{m}t(\tau,Q),
\]
and that
\[
 \Big(1-\frac{m}{n}\Big)^m\ \leq\ \frac{m!}{n^m}\binom{n}{m}\ \leq\ 1.
\]
On the other hand, it is clear that
\[
 \lambda([0,1]^m_<\backslash[0,1]^m_{<,good})\ \leq\
n\binom{m}{2}\frac{1}{n^2}\ \leq\ \frac{m^2}{n}.
\]
Then,
\[
 \frac{m!}{n^m}\binom{n}{m}t(\tau,Q)\ \leq\ t(\tau,Z)\ \leq\
t(\tau,Q)\ +\ m!\frac{m^2}{n}
\]
and the result follows.
\end{proof}

The next step is to show that the sequence $(Z_j(x,y))_{j \in \mathbb{N}}$ converges almost everywhere. The desired function is then built upon this limit.

We now follow this plan. First, we state a result that gives us a convenient subsequence $(\sigma_m')$ of $(\sigma_n)$, whose proof lies in Section~\ref{provalemaLS.5.1}.
\begin{lemma}\label{lemaLS.5.1}
Every permutation sequence $(\sigma_n)_{n \in \mathbb{N}}$ with $\lim_{n \to \infty} |\sigma_n|=\infty$ has a subsequence $(\sigma'_m)_{m \in \mathbb{N}}$, $|\sigma_m'| \geq m$, for which there exist a sequence of positive integers $(k_m)_{m \in \mathbb{N}}$ with $\lim_{m \to \infty}k_m=\infty$ and a sequence of weighted permutations $(Q_m)_{m \in \mathbb{N}}$, $Q_m:[k_m]^2\to[0,1]$, satisfying the following properties.
\begin{itemize}
\item[(i)] If $j<m$, then $k_j$ divides $k_m$ and $Q_j=\widehat{Q}_{m,j}$, where
$\widehat{Q}_{m,j}$ is the matrix obtained from $Q_m$ by merging its entries in $k_j$ consecutive blocks of size $k_m/k_j$, and by replacing each block by a single value, the arithmetic mean of the entries in that block.
\item[(ii)] For every $j<m$, $\sigma'_m$ has a weak $(1/j)$-regular $k_j$-partition $P_{m,j}$ whose partition matrix $Q_{m,j}$ has dimension $k_j\times k_j$ and satisfies
$$d_{\square}(Q_{m,j},Q_j)<1/j.$$
Moreover, for $1\leq i<j\leq m$, $P_{m,j}$ refines $P_{m,i}$.  
\end{itemize}
\end{lemma}

We now state a few properties of the permutation subsequence given by Lemma~\ref{lemaLS.5.1}. For $1 \leq j \leq m$, let $R_{m,j}=\mathcal{K}(P_{m,j},Q_{m,j})$ and $S_{m,j}=\mathcal{K}(P_{m,j},Q_j)$, both of which are square matrices of order $|\sigma_m'|$, since the blow-up matrix $\mathcal{K}(P,Q)$ is obtained from $Q$ and from the partition $P=\bigcup_r C_r$ by replacing each single entry $(i,j)$ by a block of size $|C_i||C_j|$ assuming the same value.

\begin{claim}\label{eq13b}
The weighted permutations $Q_{\sigma'_m}$ and $S_{m,j}$ satisfy
$$\displaystyle{d_{\square}(Q_{\sigma'_m},S_{m,j}) \leq \frac{3}{j}+\frac{2}{k_j}+\frac{2}{|\sigma_m'|}}.$$
\end{claim}

Let $Z_{m,j}$ and $Z_j$ be the step functions (see equation~(\ref{eqdefZ})) associated with the weighted permutations $S_{m,j}$ and $Q_j$, respectively.
\begin{claim}\label{claim_UB2}
Given $\eps>0$ and a permutation $\tau$, if $m>j$ are both sufficiently large, then $$\displaystyle{\Big|t(\tau,Z_{m,j})-t(\tau,Z_j)\Big| < \frac{\eps}{8}}.$$
\end{claim}

The proofs of these claims are given in Appendix~\ref{sec.claim.proofs}.

Our next result establishes the convergence of the sequence of functions $(Z_j)_{j \in \mathbb{N}}$ to a limit permutation. We use the concept of \emph{weak convergence} $F_n \overset{w}{\longrightarrow} F$ of a sequence of functions $(F_n)_{n \in \mathbb{N}}$ to a function $F$. This is pointwise convergence of $F_n$ to $F$ at all points at which $F$ is continuous. The use of this type of convergence is convenient because, as a consequence of classical theorems such as the Helly--Bray Theorem and Alexandrov's Pormanteau Theorem (see Section~\ref{provaLS.5.2} for statements and references), the following holds: given a sequence $(F_n)_{n \in \mathbb{N}}$ of cumulative density functions that converges weakly to a cumulative density function $F$, the limit, as $n$ tends to infinity, of the integrals of a given function $g$ over the measures induced by the $F_n$ is equal to the integral of $g$ over the measure induced by $F$.  
\begin{lemma}\label{lemaLS.5.2}
Let $(k_m)_{m \in \mathbb{N}}$ be a sequence of positive integers such that $\lim_{m \to \infty}k_m=\infty$ and let $(Q_m)_{m \in \mathbb{N}}$ be a sequence of weighted permutations $Q_m:[k_m]^2 \to [0,1]$
satisfying the conditions of Lemma~\ref{lemaLS.5.1}.
Then there exists a limit permutation $Z:[0,1]^2\to[0,1]$ such that,  
\begin{itemize}
\item[(i)] for almost all $(x,y)\in[0,1]^2$, $Z_{Q_m}(x,y)\longrightarrow Z(x,y)$ as $m$ tends to infinity;
\item[(ii)] for almost all $x\in[0,1]$, $Z_{Q_m}(x,\cdot)\overset{w}{\longrightarrow} Z(x,\cdot)$  as $m$ tends to infinity;
\item[(iii)] for every $m$ and $i,j\in[k_m]$,
$$Q_m(i,j) = k_m^2\int_{(i-1)/k_m}^{i/k_m}\int_{(j-1)/k_m}^{j/k_m}Z(x,y)~dx ~dy.$$
\end{itemize}
\end{lemma}
The proof of this result is postponed to Section~\ref{provaLS.5.2}. An important property of the limit obtained in Lemma~\ref{lemaLS.5.2} is that  subpermutation density is continuous under it, as stated in the next lemma, whose proof is in Section~\ref{provalem.QtoZ}.
\begin{lemma}\label{lem.QtoZ}
Let $(k_m)_{m \in \mathbb{N}}$ be a sequence of positive integers such that $\lim_{m \to \infty}k_m=\infty$ and let $(Q_m)_{m \in \mathbb{N}}$ be weighted permutations with the properties of Lemma~\ref{lemaLS.5.2}.
Then the function $Z$ given by Lemma~\ref{lemaLS.5.2} satisfies
$$\lim_{n \to \infty} t(\tau,Q_n)=t(\tau,Z) \textrm{ for every permutation }\tau.$$
\end{lemma}

An immediate consequence of this result is that, given $\eps>0$, there exists $j_0>0$ such that, for $j>j_0$, the step function $Z_j$ associated with $Q_j$ satisfies
\begin{equation}\label{claim_UB1}
\Big|t(\tau,Z_j)-t(\tau,Z) \Big| < \frac{\eps}{4}.
\end{equation}

To establish Theorem~\ref{teorema_principal1}, we prove that the function $Z:[0,1]^2\to [0,1]$ given by Lemma~\ref{lemaLS.5.2} is a limit of $(\sigma_n)_{n \in \mathbb{N}}$.

\begin{proof}[Proof of Theorem~\ref{teorema_principal1}] Let $(\sigma_n)_{n \in \mathbb{N}}$ be a convergent permutation sequence with $\lim_{n \to \infty} |\sigma_n|=\infty$. Let $\tau$ be a permutation and fix $\eps>0$. We wish to prove that
$$\Big|t(\tau,Z)-\lim_{n\to\infty}t(\tau,\sigma_n)\Big| \leq \varepsilon.$$

Select a subsequence $(\sigma_m')_{m \in \mathbb{N}}$ as in Lemma~\ref{lemaLS.5.1} and a corresponding limit $Z$ as in Lemma~\ref{lemaLS.5.2}. For positive integers $m>j$, the triangle inequality leads to
\begin{equation*}
\begin{split}
\Big|t(\tau,Z)-\lim_{n\to\infty}t(\tau,\sigma_n)\Big| \leq & \Big|t(\tau,\sigma'_m)-\lim_{n\to\infty}t(\tau,\sigma_n)\Big| + \Big|t(\tau,S_{m,j})-t(\tau,\sigma_m')\Big|\\
&+ \Big|t(\tau,Z_j)-t(\tau,S_{m,j})\Big| +  \Big|t(\tau,Z)-t(\tau,Z_{j})\Big|,
\end{split}
\end{equation*}
where $S_{m,j}$ and $Z_j$ are defined in the discussion preceding this proof. We now bound each of the terms on the right-hand side. Since $(\sigma_n)_{n \in \mathbb{N}}$ is convergent, the real sequence $(t(\tau,\sigma_n))_{n \in \mathbb{N}}$ converges. The subsequence $(t(\tau,\sigma'_m))_{m \in \mathbb{N}}$ must converge to the same limit, and we have
$$\Big|t(\tau,\sigma'_m)-\lim_{n\to\infty}t(\tau,\sigma_n)\Big| < \frac{\varepsilon}{4}$$
for $m>m_0>0$ sufficiently large.
Now, assuming that $|\sigma'_m|\geq m>\min\{j,k_j\}>m_0>56 |\tau|^2/\eps$, Lemma~\ref{lema.cut.subpermut} and Claim~\ref{eq13b} lead to
\begin{equation*}
\Big|t(\tau,Q_{\sigma'_m})-t(\tau,S_{m,j})\Big|
\leq 2|\tau|^2\cdot d_{\square}(Q_{\sigma'_m},S_{m,j})
\leq 2|\tau|^2 \left(\frac{3}{j}+\frac{2}{k_j}+\frac{2}{|\sigma_m'|}\right)
\leq \frac{\varepsilon}{4}.
\end{equation*}
Also, Lemma~\ref{lema_auxiliar} tells us that $|t(\tau,Z_{m,j})-t(\tau,S_{m,j})|<\eps/8$ for $m$ sufficiently large, which, combined with Claim~\ref{claim_UB2}, implies that
$$\Big|t(\tau,Z_j)-t(\tau,S_{m,j})\Big| \leq \Big|t(\tau,Z_j)-t(\tau,Z_{m,j})\Big|+\frac{\eps}{8}  < \frac{\varepsilon}{4}$$
for $m>j>m_0$ sufficiently large. Finally, equation~(\ref{claim_UB1}) states that
$$\Big|t(\tau,Z)-t(\tau,Z_{j})\Big| < \frac{\varepsilon}{4}$$
for $j>m_0$ sufficiently large. This concludes the proof of Theorem~\ref{teorema_principal1}. \end{proof}

\section{$Z$-random permutations}\label{sec_random}

In this section, we show that every limit permutation $Z:[0,1]^2\to[0,1]$ is a limit of a convergent permutation sequence. This establishes the remaining part of Theorem~\ref{teorema_principal}. To achieve this objective, limit permutations are used to define a model of random permutations, which we call \emph{$Z$-random permutations}. This resembles the concept of $W$-random graphs introduced by Lov\'{a}sz and Szegedy in~\cite{lovasz06}, defined for each fixed graphon $W$.

\begin{definition}[$Z$-random permutation]
Given a limit permutation $Z:[0,1]^2 \to [0,1]$ and a positive integer $n$, a \emph{$Z$-random permutation} $\sigma(n,Z)$
is a permutation of $[n]$ generated as follows. A set of $n$ pairs $(X_1,a_1),\ldots,(X_n,a_n)$ in $[0,1]^2$
is independently generated, where, for each $i$, $X_i$ is chosen uniformly and $a_i$ is chosen according to the probability
distribution induced by the cdf $Z(X_i,\cdot)$. These pairs define the functions $R,S:[n]\to[n]$,
where $R(i)= |\{j~:~X_j \leq X_i\}|$ and $S(i)= |\{j~:~a_j \leq a_i\}|$ for every $i$.
The random permutation is given by the composition $\sigma(n,Z)=S \cdot R^{-1}$, that is, $\sigma(n,Z)(i)=S(R^{-1}(i))$ for every $i\in[n]$.
\end{definition}

Observe that the functions $R$ and $S$ in the above definition captures the order of the components of $X$ and $a$, respectively. Indeed, the vector $(X_{R^{-1}(1)},\ldots,X_{R^{-1}(n)})$ consists of the components of $X$ ordered in increasing order. As an example, suppose that $n=4$, $X=(0.7,0.3,0.9,0.2)$ and $a=(0.8,0.1,0.5,0.3)$.
Then the permutations $R$ and $S$ are given by $R=(3,2,4,1)$ and $S=(4,1,3,2)$, so that $\sigma=S\cdot R^{-1}=(2,1,4,3)$. 
Note that $Z$-random permutations are well-defined with probability 1, because the probability that either $(X_1,\ldots,X_n)$ or $(a_1,\ldots,a_n)$ has repeated elements is zero. For the former, this is by definition of the uniform
probability space and, for the latter, this is a consequence of the forthcoming Lemma~\ref{lema_uniform}.

Observe that, as mentioned in the introduction, one could alternatively generate $Z$-random permutations by first generating a sequence $X_1<\cdots<X_n$ uniformly in $[0,1]^n_<$
and then drawing each $a_i$ in $[0,1]$ independently according to the probability distribution induced by $Z(X_i,\cdot)$.
The random permutation is given by the order of the elements in $(a_1,\ldots,a_n)$.

\begin{theorem}[$t(\tau,\sigma(n,Z))$ converges to $t(\tau,Z)$]\label{rud.teoZPerm}
Let $Z$ be a limit permutation and let $n\leq m$ be positive integers. For a permutation $\tau:[m]\to[m]$ and a constant $\eps>0$, we have
\begin{itemize}
\item[(a)] $\EE\Big(t(\tau,\sigma(n,Z))\Big)\ =\ t(\tau,Z),$
\item[(b)] $\PP\Big(\Big|t(\tau,\sigma(n,Z))-t(\tau,Z)\Big|>\varepsilon\Big)\ \leq\ 2\exp\{-\varepsilon^2n/2m^2\}.$
\end{itemize}
\end{theorem}

\begin{proof} Let $\tau:[m]\to[m]$ and $n>m$ be fixed throughout the proof and let $X=(X_1,\ldots,X_n)$ be the random vector chosen in the definition of the random permutation $\sigma=\sigma(n,Z)$. Given an injective function $\phi:[m] \to [n]$ for which $\phi(1)<\cdots<\phi(m)$,
consider the event $A_{\phi}$ that the elements of $\sigma$ indexed by $\phi$ form a subpermutation $\tau$.
In other words $A_{\phi}$ is the event that, for every $i,j \in [m]$,
\begin{equation}\label{eq_aux_1}
\sigma(\phi(i))<\sigma(\phi(j)) \textrm{ if and only if }\tau(i) < \tau(j).
\end{equation}
With this, since $\Lambda(\tau,\sigma)$ denotes the number of occurrences of $\tau$ in $\sigma$, we have
$$\EE\Big(t(\tau,\sigma(n,Z))\Big)=\binom{n}{m}^{-1}\EE\Big(\Lambda(\tau,\sigma(n,Z))\Big) =\binom{n}{m}^{-1}\sum_\phi \EE(A_\phi) = \binom{n}{m}^{-1}\sum_\phi \PP(A_\phi).$$
To compute $\PP(A_\phi)$ consider $n$ i.i.d.\,uniform random variables $X_1,\ldots,X_n$ on $[0,1]$ and let $B_\phi$ be the event that the ranks of $X_1',\ldots,X'_m$ are $\phi(1),\ldots,\phi(m)$, respectively. If we condition on $B_\phi$, the distribution of $X_1',\ldots,X'_m$ is precisely the distribution of $X_{R^{-1}(\phi(1))},\ldots,X_{R^{-1}(\phi(m))}$. By the independence and the uniformity of our random choices, we have $\PP(B_\phi)=m!\binom{n}{m}$, so that
\begin{equation*}
\begin{split}
\PP(A_\phi)&=\EE\Big(L_{\tau,Z}\left(X_{R^{-1}(\phi(1))},\ldots,X_{R^{-1}(\phi(m))}\right)\Big)\\
&=\EE\Big(L_{\tau,Z}\left(X_{1}',\ldots,X_{m}'\right)~|~B_\phi\Big)=m!\binom{n}{m} \EE\Big(L_{\tau,Z}\left(X_{1}',\ldots,X_{m}'\right)\Big) \mathbf{1}(X_{1}'<\cdots<X_{m}'),
\end{split}
\end{equation*}
where $\mathbf{1}(A)$ denotes the indicator random variable for the event $A$.

Summing over all $\phi$ and using the fact that $\sum_\phi \mathbf{1}(B_\phi)= \mathbf{1}(X_{1}'<\cdots<X_{m}')$, we obtain
\begin{equation*}
\begin{split}
\binom{n}{m}^{-1} \sum_\phi \PP(A_\phi) &= m! \EE\Big(L_{\tau,Z}\left(X_{1}',\ldots,X_{m}'\right)\Big) \mathbf{1}(X_{1}'<\cdots<X_{m}')\\
&=\int_{x\in[0,1]^m_<}\Big(\int_{[0,1]^m_<} dZ(x_{\tau^{-1}(1)},\cdot)\cdots dZ(x_{\tau^{-1}(m)},\cdot)\Big)\ dx_1 \ldots dx_m=t(\tau,Z),
\end{split}
\end{equation*}
as required for part (a).

To prove part (b), we rely on exposure martingales, as defined by Alon and Spencer~(\cite{alon_spencer00}, Chapter 7).
The basic idea is to reveal the pairs of random choices $(X_1,a_1),\ldots,(X_n,a_n)$ one by one, analyzing the influence of
each new random choice on the number of occurrences of $\tau$ as a subpermutation of the final $Z$-random permutation $\sigma(n,Z)$.

Let $Y_0,\ldots Y_n$ be the random variables given by $Y_i=\EE\left(t(\tau,\sigma)\ |\ (X_1,a_1),\ldots,(X_i,a_i)\right)$.
In particular, by part (a), $Y_0=\EE\Big(t(\tau,\sigma)\Big)=t(\tau,Z)$. Note that, after revealing the first $i$ pairs
of random choices $(X_1,a_1),\ldots,(X_i,a_i)$, $i<n$, the actual values of the final permutation $\sigma(n,Z)$ are
still unknown. However, the relative order of the elements generated by these random choices is determined, so that
the number of occurrences of the subpermutation $\tau$ can be counted in an incremental way. Moreover, assuming that all pairs have been chosen, if any pair of random choices $(X_{i},a_{i})$ were assigned a new value, the number of occurrences of $\tau$ may vary by at most $\binom{n-1}{m-1}$. Thus, the variation $|Y_{i+1}-Y_i|$ is bounded above by $$\binom{n-1}{m-1} \cdot\binom{n}{m}^{-1}=\frac{m}{n}.$$
By Azuma's inequality (see for example~\cite{alon_spencer00}, Corollary 7.2.2), we have that, given $\eps>0$,
\[
  \PP\Big(\Big|t(\tau,\sigma(n,Z))-t(\tau,Z)\Big|\ >\ \varepsilon\Big)\ =\
  \PP\Big(\Big|Y_n-Y_0\Big|\ >\ \varepsilon\Big)
\]
%\[
%  =\ \PP\Big(\Big|L(\tau,\sigma(n,Z))-E(L(\tau,\sigma(n,Z)))\Big|\ >\ \varepsilon \frac{n}{2m}\Big)
%\]
\[
  \leq\ 2\exp\Big\{-\frac{\varepsilon^2}{2n(m/n)^2}\Big\}\ \leq\ 2\exp\Big\{\frac{-\varepsilon^2n}{2m^2}\Big\}.
\]
This concludes the proof of the theorem. \end{proof}

Let $Z$ be a limit permutation and $\tau$ be a permutation on $[m]$. Observe that, when we restrict Theorem~\ref{rud.teoZPerm}(a) to the case $n=m$, we obtain 
$$t(\tau,Z)=\EE\Big(t(\tau,\sigma(m,Z))\Big).$$
It is clear that $t(\tau,\sigma(m,Z))=1$ if $\tau=\sigma(m,Z)$, while $t(\tau,\sigma(m,Z))=0$ otherwise. In particular, we have
$$t(\tau,Z)=\EE\Big(t(\tau,\sigma(m,Z))\Big)=\PP\Big(\sigma(m,Z)=\tau\Big),$$
so that the permutation-based definition of subpermutation density given in the introduction coincides with Definition~\ref{def.dens.subperm.lim}.

The next corollary leads to the desired result.
\begin{cor}[$Z$-random permutations are convergent]\label{rud.corZPerm}
Given a limit permutation $Z:[0,1]^2\to[0,1]$, the permutation sequence $(\sigma(n,Z))_{n \in \mathbb{Z}}$ is convergent with probability one, and its limit is $Z$.
\end{cor}

\begin{proof} Fix $\varepsilon>0$ and a permutation $\tau:[m]\to[m]$.
For each positive integer $n$, let $A_n$ be the event that $\Big|t(\tau,\sigma(n,Z))-t(\tau,Z)\Big|>\varepsilon$.

Part (b) of Theorem~\ref{rud.teoZPerm} states that $\PP(A_n)\leq 2\exp\Big\{-\varepsilon^2n/2m^2\Big\}$, so that
\[
  \sum_{n=1}^{\infty}\PP\Big(A_n\Big)\ <\ \infty.
\]
By the Borel-Cantelli Lemma (see for example~\cite{loeve77}, Chapter V, \S 15.3), the set of values of $n$ for which the event $A_n$ occurs is finite with probability one.
Since $\varepsilon>0$ is arbitrary, we conclude that
\begin{equation}\label{eq_aux_2}
t(\tau,\sigma(n,Z))\rightarrow t(\tau,Z) \textrm{ with probability }1.
\end{equation}
Now, the set of all permutations $\tau$ is countable, hence~(\ref{eq_aux_2}) holds for every permutation $\tau$ with probability one, as required. \end{proof}

In particular, for every limit permutation $Z:[0,1]^2\to[0,1]$, there is a permutation sequence $(\sigma_n)_{n \in \mathbb{N}}$ convergent to it, namely a permutation sequence generated according to the model of $Z$-random permutations.

\section{Uniqueness}\label{sec_uniqueness}

This section is devoted to the proof of Theorem \ref{res_princ1}. The spirit of our proof is the same as in the case of graphs~\cite{borgs06b}, where uniqueness was captured by a natural pseudometric induced by a distance between graphons. We adapt this proof technique to the framework of permutations, as we introduce a metric  $d_{\square}$ in the set of limit permutations and we prove that two limit permutations are limits of the same permutation sequence if and only if $d_{\square}(Z_1,Z_2)=0$. This metric is just a generalization to limit permutations of the rectangular distance for permutations of Definition~\ref{def.rec}.

\begin{definition}[Rectangular distance between limit permutations]
Given functions $Z_1,Z_2:[0,1]^2\to[0,1]$ such that $Z(x,\cdot)$ is a cdf for every $x\in[0,1]$, the \emph{rectangular distance} between $Z_1$ and $Z_2$ is defined by
\begin{equation}
  d_{\square}(Z_1,Z_2)\ =\ \sup_{\substack{x_1<x_2\in [0,1]\\\alpha_1<\alpha_2\in[0,1]}}\
  \Big|\int_{x_1}^{x_2}\int_{\alpha_1}^{\alpha_2}dZ_1(x,\cdot)dx
  -\int_{x_1}^{x_2}\int_{\alpha_1}^{\alpha_2}dZ_2(x,\cdot)dx\Big|.
\end{equation}
\end{definition}

To simplify notation, we shall write
\[
  \int_{\alpha_1}^{\alpha_2}dZ_1(x,\cdot)dx
  -\int_{\alpha_1}^{\alpha_2}dZ_2(x,\cdot)dx =\ \int_{\alpha_1}^{\alpha_2}d\big(Z_1-Z_2\big)(x,\cdot)dx.
\]
It is not hard to see that, for permutations $\sigma_1$ and $\sigma_2$ on $[n]$, we have $d_{\square}(\sigma_1,\sigma_2)=d_{\square}(Q_{\sigma_1},Q_{\sigma_2})=d_{\square}(Z_{\sigma_1},Z_{\sigma_2})$, which allows us to extend the definition of rectangular distance to permutations on different sets of integers. Indeed, we may define $d_{\square}(\sigma,\pi):=d_{\square}(Z_{\sigma},Z_{\pi})$ for every pair of permutations $\sigma$ and $\pi$. Moreover, we may define the distance between a permutation $\sigma$ and a limit permutation $Z$ by $d_{\square}(\sigma,Z):=d_{\square}(Z_{\sigma},Z)$.

We shall demonstrate Theorem \ref{res_princ1} in the following form.
\begin{theorem}\label{cor.ida.e.volta}
The following properties hold for limit permutations $Z_1,Z_2:[0,1]^ 2\to[0,1]$.
\begin{itemize}

\item[(i)] $d_{\square}(Z_1,Z_2)=0$ if and only if $Z_1=Z_2$ almost everywhere.

\item[(ii)] A permutation sequence $(\sigma_n)_{n \in \mathbb{N}}$ converges to a limit permutation $Z$ if and only if
$$\lim_{n \to \infty} d_{\square}(Z_{\sigma_n},Z)=0.$$
\end{itemize}
In particular, if $(\sigma_n)_{n \in \mathbb{N}}$ is a permutation sequence converging to $Z_1$, then $Z_2$ is a limit to $(\sigma_n)_{n \in \mathbb{N}}$ if and only if $Z_1=Z_2$ almost everywhere.
\end{theorem}

\begin{proof}
We start with part (i). It is not hard to see that, if $Z_1$ and $Z_2$ are equal almost everywhere, then their rectangular distance is zero. For the converse, let $Z_1$ and $Z_2$ be limit permutations such that $d_{\square}(Z_1,Z_2)=0$. This implies that, for every rectangle $R=[x_1,x_2] \times [y_1,y_2] \in [0,1]^2$, we have
\begin{equation}\label{eqaux3}
\int_{x_1}^{x_2} \int_{y_1}^{y_2} dZ_1(x,\cdot) dx = \int_{x_1}^{x_2}
\int_{y_1}^{y_2} dZ_2(x,\cdot) dx.
\end{equation}
By definition of the measure induced by $Z_1(x,\cdot)$, the inner integral on the left-hand side of the previous equation is equal to
$$\int_{y_1}^{y_2} dZ_1(x,\cdot) dx= \lim_{y \to y_2^+} Z_1(x,y) - Z_1(x,y_1),$$
where, to avoid considering different cases, we make the convention that $\lim_{y \to 1^{+}} Z_1(x,y)=1$ for $y_2=1$. An analogous expression can be given for the inner integral on the right-hand side of the previous equation. Thus, rearranging the terms in equation (\ref{eqaux3}), we obtain
$$
\int_{x_1}^{x_2} \left( \lim_{y \to y_2^+} Z_1(x,y)- \lim_{y \to y_2^+}
Z_2(x,y)\right)~dx
= \int_{x_1}^{x_2} \left(Z_1(x,y_1)-Z_2(x,y_1) \right)~dx
$$
for every choice of $y_1 \leq y_2$ and $x_1 \leq x_2$. In particular, we may take $y_2=1$ and deduce that, given $y_1 \in [0,1]$,
$$\int_{x_1}^{x_2} \left(Z_1(x,y_1)- Z_2(x,y_1)\right) dx=0$$
for every interval $I=[x_1,x_2] \subseteq [0,1]$. As the function $f_{y_1}:x \to Z_1(x,y_1)- Z_2(x,y_1)$ is Lebesgue measurable and bounded in $[0,1]$, this implies that $Z_1(x,y_1)=Z_2(x,y_1)$ for almost all $x \in [0,1]$.

Now, given $\eps>0$, let $A(\eps)=\{ (x,y) : |Z_1(x,y)-Z_2(x,y)| \geq \eps \}$. This set is Lebesgue measurable in $[0,1]^2$, as $Z_1-Z_2$ is measurable. Moreover, we know from our previous work that, for any fixed $y \in [0,1]$, the vertical section $A^y(\eps)=\{x:(x,y) \in A_\eps\}$ has measure zero. By Fubini's Theorem for null sets (see Oxtoby~\cite{oxtoby}, Theorem~14.2 and the discussion following Theorem~14.3), this implies that $A(\eps)$ has Lebesgue measure zero in $[0,1]^2$, so that the set
$$A=\{ (x,y) : Z_1(x,y) \neq Z_2(x,y)\}  \subseteq \bigcup_{n=1}^\infty A(1/n)$$
has measure zero, as claimed.

Our proof of part (ii) is based on the following three facts, which will be proved in Lemma~\ref{lema.amostra}, in Proposition~\ref{prop_limit2} and in the discussion proceeding it. 
\begin{itemize}
\item[(a)] The convergence of the random sequence $(\sigma(n,Z))_{n \in \mathbb{N}}$ to $Z$ is capured by $d_{\square}$. Indeed, with probability 1, we have $\lim_{n \to \infty} d_{\square}(Z_{\sigma(n,Z)},Z)=0$.

\item[(b)] If $(\sigma_n)_{n \in \mathbb{N}}$ and $(\pi_n)_{n \in \mathbb{N}}$ are permutation sequences converging to the same limit permutation $Z$, where $|\sigma_n|=|\pi_n|$ for every $n$, then $\lim_{n \to \infty} d_{\square}(\sigma_n,\pi_n) = 0$.

\item[(c)] If $(\sigma_n)_{n \in \mathbb{N}}$ is a permutation sequence converging to $Z_1$, then $Z_2$ is a limit to $(\sigma_n)_{n \in \mathbb{N}}$ if and only if $d_{\square}(Z_1,Z_2)=0$.
\end{itemize}

First observe that the desired result follows. Recall that we have to show that a permutation sequence $(\sigma_n)_{n \in \mathbb{N}}$ converges to a limit permutation $Z$ if and only if
$$\lim_{n \to \infty} d_{\square}(Z_{\sigma_n},Z)=0.$$
First assume that a permutation sequence $(\pi_n)_{n \in \mathbb{N}}$ converges to a limit $Z$. Then
\begin{equation*}
\begin{split}
d_{\square}(Z_{\pi_n},Z) &\leq \EE(d_{\square}(Z_{\pi_n},Z_{\sigma(|\pi_n|,Z)}))+ \EE(d_{\square}(Z_{\sigma(|\pi_n|,Z)},Z))\\
&= \EE(d_{\square}(\pi_n,\sigma(|\pi_n|,Z)))+\EE(d_\square(\sigma(|\pi_n|,Z),Z)).
\end{split}
\end{equation*}
The first term goes to zero by part (b), since the sequences $(\sigma(|\pi_n|,Z))_{n \in \mathbb{N}}$ and $(\pi_n)_{n \in \mathbb{N}}$ converge to the same limit with probability one. The second term goes to zero by part (a).

Conversely, let $(\pi_n)_{n \in \mathbb{N}}$ be such that $d_{\square}(Z_{\pi_n},Z)$ tends to zero as $n$ tends to infinity. Let $(\pi_n')$ be a convergent subsequence of $(\sigma_n)$ with limit $Z_1$. Then, since
\begin{equation*}
d_{\square}(Z,Z_1) \leq d_{\square}(Z_1,Z_{\pi_n'}) + d_{\square}(Z,Z_{\pi_n'})
\end{equation*}
for every $n$, and both terms on the right-hand side tend to zero as $n$ tends to infinity, we conclude that $d_{\square}(Z,Z_1)=0$. Part (c) tells us that $Z$ is also a limit for $(\pi_n')_{n \in \mathbb{N}}$.

We have proved that every convergent subsequence of $(\pi_n)$ converges to $Z$. The fact that the sequence $(\pi_n)$ itself must converge to $Z$ is a consequence of Lemma~\ref{conv_subseq}, which states that the space of all permutations is sequentially compact.
\end{proof}

We now prove the above assertions (a), (b) and (c). 
\begin{lemma}\label{lema.amostra}
Let $Z$ be a limit permutation. With probability one, the sequence of $Z$-random permutations $(\sigma(n,Z))_{n \in \mathbb{N}}$ satisfies
\[
  \lim_{n\to\infty} d_{\square}(Z,Z_{\sigma(n,Z)})\ =\ 0.
\]
\end{lemma}

For simplicity, we refer to the rectangular distance between $Z$ and $Z_{\sigma(n,Z)}$ as the distance between $Z$ and the \emph{sample} $\sigma(n,Z)$ of $Z$, which is denoted by $d_{\square}(Z,\sigma(n,Z))$. In the proof of Lemma~\ref{lema.amostra}, we shall determine that the components of the vector $(a_1,\ldots,a_n)$ in the definition of a random permutation are uniformly distributed in $[0,1]$. We state this result here for later reference.
\begin{lemma}\label{lema_uniform}
Let $Z$ be a limit permutation and $(X_1,\ldots,X_n)$ be a random variable such that each $X_i$ is chosen uniformly at random in the interval $[0,1]$. Consider the random variable $(a_1,\ldots,a_n)$ where, for $i \in [n]$, the component $a_i$ is chosen in $[0,1]$ according to the probability distribution given by $Z(x_i,\cdot)$. Then
\[
\PP(\alpha_1 \leq a_k \leq \alpha_2)  =\int_0^1\int_{\alpha_1}^{\alpha_2} dZ(x)dx\ =\
  \int_0^1\Big(\lim_{y \to \alpha_2^+}Z(x,y)-Z(x,\alpha_1)\Big)dx\ =\ \alpha_2-\alpha_1.
\]
In particular, the components of the vector $(a_1,\ldots,a_n)$ are uniformly distributed in $[0,1]$.
\end{lemma}

\begin{proof}[Proof of Lemma~\ref{lema.amostra}] Consider the random permutations $\sigma_n=\sigma(n,Z)$. We have to show that, for every $\eps>0$, there exists a positive integer $n_0$ such that, for every $n >n_0$, every $x_1 \leq x_2 \in [0,1]$ and every $\alpha_1 \leq \alpha_2 \in [0,1]$, we have
$$\left|\int_{x_1}^{x_2}
  \int_{\alpha_1}^{\alpha_2}d(Z-Z_{\sigma_n})(x,\cdot)\ dx \right|<\eps.$$

For $\varepsilon>0$ fixed, we shall consider the event $B_n=B_n(\eps)$ given by
\[
  B_n(\eps)\ =\ \left(\exists~ x_1,x_2,\alpha_1,\alpha_2\in V_n,\ \
  \Big|\int_{x_1}^{x_2}\int_{\alpha_1}^{\alpha_2}d\big(Z_{\sigma_n}-Z\big)(x,\cdot)dx\Big|
  > \varepsilon+6n^{-1/4}\right),
\]
where  $V_n=\left\{0,\frac{1}{n},\frac{2}{n},\ldots,\frac{n-1}{n},1\right\}$. We shall prove that, for $n$ sufficiently large,
\begin{equation}\label{eqaux4}
\PP\big(\ B_n(\eps)\ \big)\  \leq\ e^{-\varepsilon^2n/2+4\ln(n)}.
\end{equation}
As the sum $\sum_{n=1}^\infty e^{-\varepsilon^2n/2+4\ln(n)}$ converges, the  Borel-Cantelli Lemma implies that, for every $\eps>0$, the set of values of $n$ for which the event $B_n(\eps)$ occurs is finite with probability one. Moreover, Lemma~\ref{lema_uniform} implies that 
$$\int_{0}^1 \int_{y}^{y+1/n}d Z(x,\cdot)dx=1/n$$
for every $y \in [0,1-1/n]$ and every limit permutation $Z$. For every $x \in [0,1-1/n]$, the definition of limit permutation also leads to
$$\int_{x}^{x+1/n} \int_{0}^{1}d Z(x,\cdot)dx=1/n.$$
Because the real numbers $x_1,x_2,\alpha_1,\alpha_2$ are each within $1/n$ of a rational number in the set $B_n(\eps)$, we may use these two equations to conclude that, for $n$ sufficiently large,
\begin{equation}\label{eq.borelcanteli2}
  \PP\Big(d_{\square}(Z,\sigma(n,Z))>\varepsilon+7n^{-1/4}\Big)\ \leq\
  e^{-\varepsilon^2n/2+4\ln(n)}.
\end{equation}
Since $\eps>0$ is arbitrary, we deduce that $d_{\square}(Z,\sigma(n,Z))\to 0$ with probability one as $n$ tends to infinity, which is the desired result.

We henceforth concentrate on establishing (\ref{eqaux4}). (We also prove Lemma~\ref{lema_uniform} along the way.) To this end, let $x_1 \leq x_2\in[0,1]$ and $\alpha_1 \leq \alpha_2\in[0,1]$ be fixed and consider
$$L(\sigma_n)=n \cdot \int_{x_1}^{x_2} \int_{\alpha_1}^{\alpha_2}dZ_{\sigma_n}(x,\cdot)\ dx=n \cdot\int_{x_1}^{x_2} \left(\lim_{y \to \alpha_2^+} Z_{\sigma_n}(x,y)- Z_{\sigma_n}(x,\alpha_1) \right)~dx.$$
Recall that the step function $Z_{\sigma_n}$ is defined in (\ref{eqdefZ}) in terms of the bipartite adjacency matrix $Q_{\sigma_n}$ associated with the permutation $\sigma_n$. In particular, we have
\begin{equation*}
\lim_{y \to \alpha_2^+} Z_{\sigma_n}(x,y)- Z_{\sigma_n}(x,\alpha_1) =
\left\{
\begin{matrix}
1, &\textrm{ if }\lceil \alpha_1 n\rceil \leq \sigma_n(\lceil x\, n \rceil)  < \lceil \alpha_2 n\rceil, \\
0, &\textrm{ otherwise.}
\end{matrix}
\right.
\end{equation*}
This suggests that we should rewrite $L(\sigma_n)$ as
\begin{equation}\label{eq.LSigma}
\begin{split}
  L(\sigma_n) & = n \cdot \sum_{i=\lceil x_1 n \rceil}^{\lceil x_2 n \rceil-1} \int_{i/n}^{(i+1)/n} \int_{\alpha_1}^{\alpha_2}dZ_{\sigma_n}(x,\cdot)\ dx + n \cdot f_{\sigma_n}(x_1,x_2,\alpha_1,\alpha_2)\\
&=\Big| \big\{i \in [n] : \lceil x_1 n \rceil \leq i \leq \lceil x_2 n \rceil,  \lceil \alpha_1 n\rceil \leq \sigma_n(i)  < \lceil \alpha_2 n\rceil  \big\}\Big| + n \cdot f_{\sigma_n}(x_1,x_2,\alpha_1,\alpha_2),\\
\end{split}
\end{equation}
where $|f_{\sigma_n}(x_1,x_2,\alpha_1,\alpha_2)|=\left|\int_{x_1}^{\lceil x_1 n \rceil/n} \int_{\alpha_1}^{\alpha_2}dZ_{\sigma_n}(x,\cdot)\ dx  - \int_{x_2}^{\lceil x_2 n \rceil/n} \int_{\alpha_1}^{\alpha_2}dZ_{\sigma_n}(x,\cdot)\ dx \right| \leq 2/n$.

We now evaluate $N(\sigma_n)=\big| \big\{i \in [n] : \lceil x_1 n \rceil \leq i \leq \lceil x_2 n \rceil,  \lceil \alpha_1 n\rceil \leq \sigma_n(i)  < \lceil \alpha_2 n\rceil  \big\}\big|$. Let $R_n$ and $S_n$ be the permutations from which $\sigma_n$ has been generated, that is, $\sigma_n=S_n \cdot R_n^{-1}$, so that $N(\sigma_n)$ is given by the number of elements $k \in [n]$ for which $ \lceil x_1 n \rceil \leq R_n^{-1}(k) \leq \lceil x_2 n \rceil$ and $ \lceil \alpha_1 n\rceil  \leq S_n \cdot R_n^{-1} (k) < \lceil \alpha_2 n\rceil$.  

For every $k\in[n]$, let $A_k$ be the event
$$A_k=\Big( \lceil x_1 n \rceil \leq R^{-1}_n(k) \leq \lceil x_2 n \rceil  \textrm{ and } \lceil \alpha_1 n\rceil  \leq S_n \cdot R^{-1}_n(k) < \lceil \alpha_2 n\rceil\Big),$$
and let $I_k$ denote the indicator random variable for $A_k$. On the one hand, we clearly have $p=\PP(A_i)=\PP(A_j)$ for every $i$ and $j$, as the pairs $(X_1,a_1),\ldots,(X_n,a_n)$ are generated independently according to the same rule. On the other hand, the above implies that $N(\sigma_n)=\sum_{k\in[n]}I_{k}$, so that, by linearity of expectation,
\begin{equation}\label{eqesperanca}
  \EE\Big(N(\sigma_n)\Big)\ =\ \sum_{k\in[n]}\EE(I_k)\ =\ n\cdot p.
  %\ =\ n\cdot \int_{x_1}^{x_2}\int_{\alpha_1}^{\alpha_2}dZ(x)\ dx.
\end{equation}

To estimate $p$, we let $k$ be fixed and consider the probability that $A_k$ holds. This is the probability that $X_k$ is ranked between positions $\lceil x_1 n \rceil$ and $\lceil x_2 n \rceil$, while $a_k$ is ranked between positions $\lceil \alpha_1 n \rceil$ and $\lceil \alpha_2 n \rceil$, among all $X_i$ and $a_i$ that are randomly generated.

It is a well-known fact that, if $n$ elements are independently chosen according to the uniform distribution in $[0,1]$, then the $k$-th smallest element has distribution Beta$(k,n-k+1)$ for every $k\in[n]$, so that it has mean $k/(n+1)$ and variance smaller than $1/n$. In our case, the elements  $(X_1,\ldots,X_n)$ are uniformly and independently distributed by definition. Perhaps more surprisingly, this is also the case for the elements
$(a_1,\ldots,a_n)$,
since we have
\[
\PP(\alpha_1 \leq a_k \leq \alpha_2)  =\int_0^1\int_{\alpha_1}^{\alpha_2} dZ(x)dx\ =\
  \int_0^1\Big(\lim_{y \to \alpha_2^+}Z(x,y)-Z(x,\alpha_1)\Big)dx\ =\ \alpha_2-\alpha_1
\]
by part (b) of Definition~\ref{def.perm.lim}. Observe that this establishes Lemma~\ref{lema_uniform}.

With Chebyschev's inequality and $n$ sufficiently large, we conclude that the $\lceil x_1n\rceil$-th smallest element of $(X_1,\ldots,X_n)$ is concentrated around $x_1$ so that the variation is larger than $n^{-1/4}$ with probability smaller than $n^{-1/2}$. This type of concentration also holds for the $\lceil x_2n\rceil$-th smallest element of $(X_1,\ldots,X_n)$, and for the the $\lceil \alpha_1n\rceil$-th  and $\lceil \alpha_2n\rceil$-th smallest elements of $(a_1,\ldots,a_n)$, which are concetrated around $x_2$, $\alpha_1$ and $\alpha_2$, respectively.

Let $D$ denote the event that the $\lceil x_1n\rceil$-th and the $\lceil x_2n\rceil$-th smallest elements of $(X_1,\ldots,X_n)$ and that the
$\lceil\alpha_1n\rceil$-th and the $\lceil\alpha_2n\rceil$-th smallest elements of $(a_1,\ldots,a_n)$ are within $n^{-1/4}$ of $x_1,x_2,\alpha_1,\alpha_2$, respectively. Clearly, $\PP(D)\geq 1-4n^{-1/2}$ and, with conditional probabilities, we have
\begin{equation}\label{eq.semnome2}
  \PP\Big(A_k\Big)=\PP\Big(A_k\ |\ D\Big)\cdot\PP(D)\ +\ \PP\Big(A_k\ |\ \bar{D}\Big)\cdot\PP(\bar{D}).
\end{equation}

Consider the quantities $x_{1}^{-}=x_1-n^{-1/4}$, $x_{1}^{+}=x_1+n^{-1/4}$, $x_{2}^{-}=x_2-n^{-1/4}$, $x_{2}^{+}=x_2+n^{-1/4}$,
$\alpha_{1}^{-}=\alpha_1-n^{-1/4}$, $\alpha_{1}^{+}=\alpha_1+n^{-1/4}$, $\alpha_{2}^{-}=\alpha_2-n^{-1/4}$
and $\alpha_{2}^{+}=\alpha_2+n^{-1/4}$. We define
events $E_1$ and $E_2$ as follows.
\[
  E_1=\Big(x_{1}^{+}<X_k\leq x_{2}^{-}\ \ \textrm{and}\ \ \alpha_{1}^{+}<a_k\leq\alpha_{2}^{-}\Big)
\]
%  \ \ \ \ and\ \ \ \
\[
  E_2=\Big(x_{1}^{-}<X_k\leq x_{2}^{+}\ \ \textrm{and}\ \ \alpha_{1}^{-}<a_k\leq\alpha_{2}^{+}\Big).
\]

By our work so far, we know that $\PP\big(A_k|D\big)\geq\PP\big(E_1|D\big)$ and $\PP\big(A_k|D\big)\leq\PP\big(E_2|D\big)$. With the notation $a=b \pm \eps$ standing for $a \in [b-\eps,b+\eps]$, note that
\begin{equation*}
\begin{split}
  \PP(E_1)&=\int_{x_1^+}^{x_2^-}\int_{\alpha_1^+}^{\alpha_2^-}dZ(x)dx =\int_{x_1}^{x_2}\int_{\alpha_1}^{\alpha_2}dZ(x)dx-\int_{x_1}^{x_2}\int_{\alpha_1}^{\alpha_1^{+}}dZ(x)dx \\
&\, \, -\int_{x_1}^{x_2}\int_{\alpha_2^{-}}^{\alpha_2}dZ(x)dx-\int_{x_1}^{x_1^{+}}\int_{\alpha_1^{+}}^{\alpha_2^{-}}dZ(x)dx-\int_{x_2^{-}}^{x_2}\int_{\alpha_1^{+}}^{\alpha_2^{-}}dZ(x)dx\\
&=\int_{x_1}^{x_2}\int_{\alpha_1}^{\alpha_2}dZ(x)dx\ \pm \ 4n^{-1/4},
\end{split}
\end{equation*}
since we have, for example,
\begin{equation}\label{eq.semnome4}
  \int_{x_1}^{x_2}\int_{\alpha_1}^{\alpha_1^+}dZ(x)dx\ \leq\
  \int_0^1\int_{\alpha_1}^{\alpha_1^+}dZ(x)dx\ = \alpha_1^+-\alpha_1\ =\ n^{-1/4}.
\end{equation}
The same estimate holds for $\PP(E_2)$. Further observe that
$$\PP(E_1|D)=\PP(E_1,D)/\PP(D)=(\PP(E_1)-\PP(E_1 \cap \bar{D}))/\PP(D)\geq
(\PP(E_1)-\PP(\bar{D}))/\PP(D)$$
and that $\PP(E_2|D)=\PP(E_2 \cap D)/\PP(D)\leq\PP(E_2)/\PP(D)$.
With $n$ chosen sufficiently large and $\PP(\bar{D})\leq 4n^{-1/2}$, this implies that
\[
  \PP\Big(E_1\ |\ D\Big)\ \geq\ \PP(E_1)\ -\ 5n^{-1/2}\ \ \ \ \ \textrm{and}\ \ \ \ \
  \PP\Big(E_2\ |\ D\Big)\ \leq\ \PP(E_2)\ +\ 5n^{-1/2}.
\]
From (\ref{eq.semnome2}), we may derive the expressions
$\PP\big(A_k\big)\ \geq\ \PP\big(E_1\ |\ D\big)-4n^{-1/2}\ \geq\ \PP(E_1)-9n^{-1/2}$ and $\PP\big(A_k\big)\ \leq\ \PP\big(E_2\ |\ D\big)+4n^{-1/2}\ \leq\ \PP(E_2)+9n^{-1/2}$, so that we have
\[
  p=\PP\Big(A_k\Big)\ =\ \int_{x_1}^{x_2}\int_{\alpha_1}^{\alpha_2}dZ(x)dx
  \ \pm\ \big(4n^{-1/4}\ +\ 9n^{-1/2}\big).
\]

To finish the proof, we now show that the value of $N(\sigma_n)$ is strongly concentrated around its expected value $\EE\Big(N(\sigma_n)\Big)$. To this end,
we use exposure martingales as in the proof of Theorem~\ref{rud.teoZPerm}. Let $Y_0,\ldots Y_n$ be the random variables given by $Y_i=\EE\left(N(\sigma_n)\ |\ (X_1,a_1),\ldots,(X_i,a_i)\right)$. With this definition, $Y_0=\EE\Big(N(\sigma_n)\Big)$ and $Y_n=\EE\Big(N(\sigma_n)|(X_1,a_1),\ldots,(X_n,a_n)\Big)=N(\sigma_n)$. Further note that $|Y_{i+1}-Y_i|\leq 1$, because, by changing the outcome of a random pair $(X_{i+1},a_{i+1})$, at most one element can be added or removed from the set $A_k$. Azuma's inequality now yields
\[
  \PP\Big(\Big|N(\sigma_n)\ -\ \EE\big(N(\sigma_n)\big)\Big|\ >\ \varepsilon n\Big)
  \ =\ \PP\Big(\Big|Y_n-Y_0\Big|>\varepsilon n\Big)
  \ \leq\ 2e^{-\varepsilon^2n/2}.
\]
For $n$ large, we have $4n^{-1/4}+9n^{-1/2}+4n^{-1}\ <\ 5n^{-1/4}$, from which we derive
\[
  \PP\Big(\Big|\int_{x_1}^{x_2}\int_{\alpha_1}^{\alpha_2}dZ_{\sigma_n}(x,\cdot)dx
  \ -\ \int_{x_1}^{x_2}\int_{\alpha_1}^{\alpha_2}dZ(x,\cdot)dx\Big|\ >\ \varepsilon+5n^{-1/4}\Big)
  \ \leq\ 2e^{-\varepsilon^2n/2}.
\]

Now, if we choose the elements $x_1,x_2,\alpha_1,\alpha_2\in[0,1]$ in the set $V_n=\left\{0,\frac{1}{n},\frac{2}{n},\ldots,\frac{n-1}{n},1\right\}$,
we have
\[
  \PP\Big(\exists~ x_1<x_2,\alpha_1<\alpha_2\in V_n,
  \Big|\int_{x_1}^{x_2}\int_{\alpha_1}^{\alpha_2}d\big(Z_{\sigma_n}-Z\big)(x,\cdot)dx\Big|
  >\varepsilon+5n^{-1/4}\Big)\leq 2\binom{n}{2}^2e^{-\varepsilon^2n/2}.
\]

This establishes equation (\ref{eqaux4}), and therefore concludes our proof.
\end{proof}

There are two important consequences of Lemma~\ref{lema.amostra}.

\begin{lemma}\label{lema.volta2}
Let $Z_1$ and $Z_2$ be limit permutations such that, for every $m>1$, the condition $\Big|t(\tau,Z_1)-t(\tau,Z_2)\Big|\leq 1/(m+1)!$ holds for every permutation $\tau$ on $[m]$. Then $d_{\square}(Z_1,Z_2)=0$.
\end{lemma}

\begin{proof}
We have to demonstrate that $d_{\square}(Z_1,Z_2)=0$ whenever $Z_1$ and $Z_2$ satisfy $|t(\tau,Z_1)-t(\tau,Z_2)|\leq 1/(m+1)!$ for every permutation $\tau : [m] \to [m]$.

By the triangle inequality, we have
$$d_{\square}(Z_1,Z_2) \leq d_{\square}(Z_1,\sigma(m,Z_1))
 + d_{\square}(\sigma(m,Z_1),\sigma(m,Z_2)) + d_{\square}(\sigma(m,Z_2),Z_2)$$
for every randomly generated $\sigma(m,Z_1)$ and $\sigma(m,Z_2)$, so that
\[
 d_{\square}(Z_1,Z_2) \leq \EE(d_{\square}(Z_1,\sigma(m,Z_1)))
 + \EE(d_{\square}(\sigma(m,Z_1),\sigma(m,Z_2))) + \EE(d_{\square}(\sigma(m,Z_2),Z_2)).
\]
We now claim that, for every $\eps>0$, there exists $m=m(\eps)$ such that each expected number in this upper bound is limited above by $\eps/3$. This clearly leads to our result.

Let $\varepsilon>0$ be fixed. By Lemma~\ref{lema.amostra}, we may choose $m$ sufficiently large so as to have
$\EE(d_{\square}(Z_1,\sigma(m,Z_1)))\leq\varepsilon/3$ and $\EE(d_{\square}(Z_2,\sigma(m,Z_2)))\leq\varepsilon/3$.

To bound the term $\EE(d_{\square}(\sigma(m,Z_1),\sigma(m,Z_2)))$, first observe that, if we set $m>3/\eps$,
\[
  \sum_{\tau:[m]\to[m]}\Big|\PP\big(\sigma(m,Z_1)=\tau\big)-\PP\big(\sigma(m,Z_2)=\tau\big)\Big|
  \ \leq\ m!\frac{1}{(m+1)!}\ \leq\ \frac{\varepsilon}{3},
\]
since $\PP\big(\sigma(m,Z_1)=\tau\big)=t(\tau,Z_1)$ for every $\tau:[m]\to[m]$.

We may now easily couple
$\sigma(m,Z_1)$ and $\sigma(m,Z_2)$ in such a way that $\sigma(m,Z_1)\not=\sigma(m,Z_2)$ with probability smaller than $\varepsilon/3$.
We deduce that
$$\EE\big(d_{\square}(\sigma(m,Z_1),\sigma(m,Z_2))\big) \leq \PP\big(\sigma(m,Z_1) \neq \sigma(m,Z_2)\big) \leq \frac{\eps}{3}.$$
Our result follows.
\end{proof}

The second consequence of Lemma~\ref{lema.amostra} is that, if $k$ is a sufficiently large constant, then the rectangular distance between a large permutation $\sigma$ and a random subpermutation of $\sigma$ with length $k$ tends to be small. To state the result, let $\sub(k,\sigma)$ denote a subpermutation of a permutation $\sigma$ with length $k$ chosen uniformly at random among all these subpermutations.

\begin{lemma}\label{cor.amostra2}
Let $k$ be a sufficiently large positive integer and let $\pi:[n]\to[n]$ be a permutation with $n>e^k$. Then we have $d_{\square}(\pi,\sub(k,\pi))\leq 7k^{-1/4}$ with probability at least $1-2 e^{-\sqrt{k}/3}$.
\end{lemma}

\begin{proof}
Let $k$ and $\pi:[n]\to[n]$ be as in the statement. Observe that one way of generating $\sub(k,\pi)$ with the right probability is to generate a $Z_\pi$-random permutation, where $Z_{\pi}$ is the step function associated with $\pi$. However, we need to condition on the event $E$ that the independent random variables $X_1,\ldots,X_k$, which are generated according to the uniform distribution in $[0,1]$, are such that no distinct $X_i$ and $X_j$ lie in a same interval $(\frac{b-1}{n},\frac{b}{n}]$ for some $b\in[n]$. This restriction is necessary because the function $Z_{\pi}$ is not a limit permutation; even though $Z_\pi$ is Lebesgue measurable in $[0,1]^2$ and $Z_\pi(x,\cdot)$ is a cdf for every $x \in [0,1]$, the mass condition fails to hold. Indeed, if $X_i$ and $X_j$ are both in $(\frac{b-1}{n},\frac{b}{n}]$, where $b \in [n]$, then the random variables $a_i$ and $a_j$ generated according to the cdf $Z_\pi(X_i,\cdot)$ and $Z_\pi(X_j,\cdot)$, respectively, are equal with probability one, as the function $Z_\pi(x,\cdot)$ satisfies $Z_{\pi}(x,\alpha)=0$ if
$\pi(\lceil x n\rceil) \geq \lceil \alpha n\rceil$ and $Z_{\pi}(x,\alpha)=1$ if $\pi(\lceil x n\rceil) < \lceil \alpha n\rceil$. In other words, it has a single discontinuity of height one for every $x$, and the point at which this discontinuity occurs depends on the interval $(\frac{b-1}{n},\frac{b}{n}]$ containing $x$.

By definition of $E$, we have $\PP(E)\geq\Big(1-\frac{k}{n}\Big)^k\geq 1-\frac{k^2}{n}\geq 1- e^{-k/2} > 1- e^{-\sqrt{k}/3}$. We may now follow the proof of Lemma~\ref{lema.amostra} step by step, always conditioning on the event $E$, to conclude through equation~(\ref{eq.borelcanteli2}) with $\varepsilon=k^{-1/4}$ that
$$\PP\Big(d_{\square}(Z_{\pi},\sigma(k,Z_{\pi}))>7k^{-1/4}\ \Big|E\Big)\leq e^{-\sqrt{k}/3}.$$
It is not hard to see that this leads to
$\PP\Big(d_{\square}(\pi,\sub(k,\pi))>7k^{-1/4}\ \Big|E\Big)\leq e^{-\sqrt{k}/3}$, so that
\[
  \PP\Big(d_{\square}(\pi,\sub(k,\pi))>7k^{-1/4}\Big)\leq e^{-\sqrt{k}/3}\ +\ (1-\PP(E)) \leq 2 e^{-\sqrt{k}/3},
\]
which establishes our result.
\end{proof}

We may now prove part (b), which is given here as a proposition.
\begin{proposition}\label{prop_limit2}
Let $(\sigma_n)_{n \in \mathbb{N}}$ and $(\pi_n)_{n \in \mathbb{N}}$ be permutation sequences converging to the same limit permutation $Z$, where $|\sigma_n|=|\pi_n|$ for every $n$. Then $\lim_{n \to \infty} d_{\square}(\sigma_n,\pi_n) = 0$.
\end{proposition}

\begin{proof}
Fix $(\sigma_n)_{n \in \mathbb{N}}$, $(\pi_n)_{n \in \mathbb{N}}$ and $Z$ as in the statement of the proposition.  By hypothesis, we know that the sequences $(t(\tau,\sigma_n))$ and $(t(\tau,\pi_n))$ converge to the same limit for every fixed permutation $\tau$. In particular, for every positive integer $m$, there is $n_0=n_0(m)$ for which the following holds: for every $n>n_0$, $|t(\tau,\sigma_n)-t(\tau,\pi_n)|<1/(m+1)!$ for every permutation $\tau$ on $[m]$.

Let $\eps>0$ and fix $m$ sufficiently large in terms of $\eps$. We may now proceed as in the proof of Lemma~\ref{lema.volta2} and write
$$d_{\square}(\sigma_n,\pi_n) \leq \EE(d_{\square}(\sigma_n,\sub(m,\sigma_n))) + \EE(d_{\square}(\sub(m,\sigma_n),\sub(m,\pi_n))) + \EE(d_{\square}(\pi_n,\sub(m,\pi_n))).$$ By Lemma~\ref{cor.amostra2}, we deduce that, with probability as close to one as wished, the the terms $d_{\square}(\sigma_n,\sub(m,\sigma_n))$ and $d_{\square}(\pi_n,\sub(m,\pi_n))$ in the previous upper bound are smaller than $\eps/3$. To show that $\EE\left(d_{\square}(\sub(m,\sigma_n),\sub(m,\pi_n))\right)$ is smaller than $\eps/3$, we choose $m \geq 3/\eps$ and we use the fact that
$$ \sum_{\tau:[m]\to[m]}\Big|\PP\big(\sigma(m,\sigma_n)=\tau\big)-\PP\big(\sigma(m,\pi_n)=\tau\big)\Big| = \sum_{\tau:[m]\to[m]} |t(\tau,\sigma_n)-t(\tau,\pi_n)|
  \ \leq\ m!\frac{1}{(m+1)!}\ \leq\ \frac{\varepsilon}{3}, $$
As a consequence, the random variables $\sub(m,\sigma_n)$ and $\sub(m,\pi_n)$ may be coupled in such a way that the probability that their outcome is distinct is at most $\eps/3$. Our result follows.
\end{proof}

To conclude this section, we now prove part (c). Let $Z_1$ and $Z_2$ be limit permutations and let $(\sigma_n)_{n \in \mathbb{N}}$ be a permutation sequence converging to $Z_1$, where $\lim_{n \to \infty}|\sigma_n|=\infty$.

If $d_{\square}(Z_1,Z_2)=0$, we know that, for any fixed $y \in [0,1]$, $Z_1(\cdot,y)$ and $Z_2(\cdot,y)$ differ only in a set of measure zero. Recall that $t(\tau,Z_1)$ is given by
$$t(\tau,Z)\ =\ m!\int_{[0,1]^m_<}\Big(\int_{[0,1]^m_<}dZ(x_{\tau^{-1}(1)},\cdot) \ \cdots \ dZ(x_{\tau^{-1}(m)},\cdot)\Big)
  \ dx_1\cdots dx_m,$$
with integration taken over the $m$-simplex $[0,1]^m_<$, the set of $m$-tuples $(y_1,\ldots,y_m)$ such that $0\leq y_1<y_2<\cdots<y_m \leq 1$. Each term $dZ(x_{\tau^{-1}(i)},\cdot)$ of the inner integral comes from the measure associated with the cdf $Z(x_{\tau^{-1}(i)},\cdot)$, and the order of the integrating factors in the product measure reflects the connection between the integration variable corresponding to $y_i$ and the measure associated with the cdf $Z(x_{\tau^{-1}(i)},\cdot)$.

With this definition, we may write $t(\tau,Z_1)-t(\tau,Z_2)$ as a telescopic sum:
\begin{equation}\label{aux2}
  \Big|t(\tau,Z_1)-t(\tau,Z_2)\Big|\ =\ m!\Big|\sum_{k=1}^{m}\Big(Y_{k}-Y_{k-1}\Big)\Big|
  \ \leq\ m!\sum_{k=1}^{m}\Big|Y_{k}-Y_{k-1}\Big|,
\end{equation}
where, for $0\leq k\leq m$,
\[
  Y_k=\int_{[0,1]^m_<}\Big(\int_{[0,1]^m_<} dZ_1(x_{\tau^{-1}(1)},\cdot) \ldots dZ_1(x_{\tau^{-1}(k)},\cdot)
  \ \ \ \ \ \ \ \ \ \ \ \ \ \ \ \ \ \ \
\]
\[
  \ \ \ \ \ \ \ \ \ \ \ \ \ \ \ \ \ \ \ \ \ \ \ \ \ \ \ \ \ \ \ \
  dZ_2(x_{\tau^{-1}(k+1)},\cdot)\cdots dZ_2(x_{\tau^{-1}(m)},\cdot)\Big)dx_1\cdots dx_m.
\]
For $1\leq k\leq m$, we now bound $|Y_{k}-Y_{k-1}|$.
Let $x^*=(x_1,\ldots,x_{k-1},x_{k+1},\ldots,x_m)\in[0,1]^{m-1}_<$, and let $Z(x^{*})$ be the measure on $[0,1]^{m-1}$ given by the product of the measures associated with the cdfs
$Z_1(x_1),\ldots,Z_1(x_{k-1}), Z_2(x_{k+1}),\ldots,Z_2(x_m)$.
Fubini's Theorem~(see, for instance,~\cite{loeve77}, Chapter II, \S 8.2) implies that
\[
  Y_k=\int_{[0,1]^{m-1}_<}\int_{[0,1]^{m-1}_<}
  \Big(\int_{x_{k-1}}^{x_{k+1}}\int_{\alpha_{\tau(k)-1}}^{\alpha_{\tau(k)+1}} dZ_1(x_{\tau^{-1}(k)},\cdot) \ dx_{k}\Big)
  \ dZ(x^*) \ dx^*,
\]
where, for convenience, we use $x_0=\alpha_0=0$ and $x_{m+1}=\alpha_{m+1}=1$. Further note that
\[
  Y_{k-1}=\int_{[0,1]^{m-1}_<}\int_{[0,1]^{m-1}_<}
  \Big(\int_{x_{k-1}}^{x_{k+1}}\int_{\alpha_{\tau(k)-1}}^{\alpha_{\tau(k)+1}} dZ_2(x_{\tau^{-1}(k)},\cdot)\ dx_{k}\Big)
  \ dZ(x^*) \ dx^*.
\]
This leads to
\[
  \Big|Y_{k}-Y_{k-1}\Big|\ =\ \Big|\int_{[0,1]^{m-1}_<}\int_{[0,1]^{m-1}_<}
  \Big(\int_{x_{k-1}}^{x_{k+1}}\int_{\alpha_{\tau(k)-1}}^{\alpha_{\tau(k)+1}}
  d\Big(Z_1-Z_2\Big)(x_{\tau^{-1}(k)},\cdot)\ dx_{k}\Big)\ dZ(x^*)\ dx^*\Big|
\]
\[
  = \int_{[0,1]^{m-1}_<}\int_{[0,1]^{m-1}_<} \int_{x_{k-1}}^{x_{k+1}}
  \lim_{y \to \alpha_{\tau(k)+1}}\Big(Z_1(x_{\tau^{-1}(k)},y)-Z_2(x_{\tau^{-1}(k)},y)\Big)\ dx_{k}\ \ dZ(x^*)\ dx^*
\]
\[- \int_{[0,1]^{m-1}_<}\int_{[0,1]^{m-1}_<}\int_{x_{k-1}}^{x_{k+1}}\Big(Z_1(x_{\tau^{-1}(k)},\alpha_{\tau(k)-1})-Z_2(x_{\tau^{-1}(k)},\alpha_{\tau(k)-1})\Big)\ dx_{k}\ dZ(x^*)\ dx^*.
\]
Now, in both terms of the above difference, the innermost integrand is zero almost everywhere, hence the integrals are equal to zero. As a consequence, we have $\lim_{n\to \infty} t(\tau,\sigma_n)= t(\tau,Z_1)=t(\tau,Z_2)$ for every permutation $\tau$, so that $Z_2$ is also a limit of the permutation sequence $(\sigma_n)$.

For the converse, suppose that both $Z_1$ and $Z_2$ are limits of the same permutation sequence $(\sigma_n)$, and let $\tau$ be a permutation on $[m]$. By definition, we have $\lim_{n\to \infty} t(\tau,\sigma_n)= t(\tau,Z_1)=t(\tau,Z_2)$, hence $\Big|t(\tau,Z_1)-t(\tau,Z_2)\Big|=0\leq 1/(m+1)!$. As $\tau$ is arbitrary, Lemma \ref{lema.volta2} leads to $d_{\square}(Z_1,Z_2)=0$. This establishes the validity of (c).

\section{Additional proofs}\label{section_technical}

In Section~\ref{section_convergence}, we have demonstrated Theorem~\ref{teorema_principal1} using a series of results that were stated without proof. These proofs are the objective of the present section. It is organized as follows. In the first subsection, we establish some analytical properties of limit permutations. Lemmas~\ref{lemaLS.5.2} and~\ref{lem.QtoZ} are then proved in Subsections~\ref{provaLS.5.2} and~\ref{provalem.QtoZ}, respectively. We conclude the section with the proofs of Lemma~\ref{lemaLS.5.1}.

\subsection{Properties of limit permutations}

The aim of this subsection is to study the set of discontinuities of the cumulative density functions generated by a limit permutation $Z=Z(x,y)$. It is known that, for every $x \in[0,1]$, the cdf $Z(x,\cdot)$ has a countable set of discontinuity points (see~\cite{loeve77}, Chapter IV). Here, we prove that, for each $\alpha\in[0,1]$, the set of points $x \in [0,1]$ for which $\alpha$ is a discontinuity of $Z(x,\cdot)$ has Lebesgue measure zero.

\begin{lemma}\label{Lema.DescNula}
Let $Z:[0,1]^2 \to [0,1]$ satisfy the properties of a limit permutation, with the exception, possibly, of the requirement that $Z(x,\cdot)$ is continuous at the point $0$ and that $Z(x,1)=1$ for every $x \in [0,1]$. Then, for every $\alpha\in[0,1]$, the set of points $x\in[0,1]$ for which the function $Z(x,\cdot)$ is discontinuous at the point $y=\alpha$ has measure zero.
\end{lemma}

\begin{proof} Let $\alpha\in[0,1]$. Since $Z(x,\cdot)$ is a monotonic and bounded function for every $x \in [0,1]$, the limit $\lim_{\alpha'\to\alpha+}Z(x,\alpha')$ exists for every $\alpha\in[0,1)$. In the case of $\alpha=1$, we may assume that $\lim_{\alpha'\to 1+}Z(x,\alpha')=1$ by extending the cdf to the real line.

By Lebesgue's Dominated Convergence Theorem (see~\cite{loeve77}, Chapter II, \S 7.2) we have
$$
  \int_0^1\lim_{\alpha'\to\alpha+}Z(x,\alpha')\ dx\ =\ \lim_{\alpha'\to\alpha+}\int_0^1 Z(x,\alpha')\ dx
  \ =\ \lim_{\alpha'\to\alpha+}\alpha'\ =\ \alpha.
$$
The second to last step is an immediate consequence of part (b) of Definition~\ref{def.perm.lim}. Thus
$$
  \int_0^1\Big(\lim_{\alpha'\to\alpha+}Z(x,\alpha')\ -\ Z(x,\alpha)\Big)\ dx\ =\ 0,
$$
and, since $\lim_{\alpha'\to\alpha+}Z(x,\alpha')-Z(x,\alpha) \geq 0$ for every $\alpha$, we conclude that the set of points $x\in[0,1]$ such that $\lim_{\alpha'\to\alpha+}Z(x,\alpha')-Z(x,\alpha)>0$  has measure zero. \end{proof}

Now, let $D_Z\subset[0,1]^2$ be the set of all pairs $(x_1,x_2)$ such that $Z(x_1,\cdot)$ and $Z(x_2,\cdot)$ have a common discontinuity. Our aim is to show that $D_Z$ has measure zero.

Since the countable union of sets with measure zero has measure zero, Lemma~\ref{Lema.DescNula} tells us that, for every $x_1\in[0,1]$, the set $D_Z(x_1)=\{x_2~:(x_1,x_2)\in D_Z\}$ has measure zero in the line. This fact suggests that $D_Z$ has measure zero in the plane, since all vertical sections have measure zero, which is indeed true because of Lemma~\ref{lema_uniform}. We now state this fact for future reference.
\begin{lemma}\label{lema.critico}
Given a limit permutation $Z$, $D_Z$ is Lebesgue measurable in $[0,1]^2$ and has measure zero.
\end{lemma}

\subsection{The proof of Lemma~\ref{lemaLS.5.2}}\label{provaLS.5.2}

We first remind the reader of the statement of Lemma~\ref{lemaLS.5.2}.

\addtocounter{section}{-3}
\addtocounter{theorem}{4}

\begin{lemma}
Let $(k_m)_{m \in \mathbb{N}}$ be a sequence of positive integers and let $(Q_m)_{m \in \mathbb{N}}$ be weighted permutations $Q_m:[k_m]^2\to[0,1]$ satisfying the conditions of Lemma~\ref{lemaLS.5.1}.
Then there exists a limit permutation $Z:[0,1]^2\to[0,1]$ such that,  
\begin{itemize}
\item[(i)] for almost all $(x,y)\in[0,1]^2$, $Z_{Q_m}(x,y)\longrightarrow Z(x,y)$ as $m$ tends to infinity;
\item[(ii)] for almost all $x\in[0,1]$, $Z_{Q_m}(x,\cdot)\overset{w}{\longrightarrow} Z(x,\cdot)$ as $m$ tends to infinity;
\item[(iii)] for every $m$ and $i,j\in[k_m]$,
$$Q_m(i,j) = k_m^2\int_{(i-1)/k_m}^{i/k_m}\int_{(j-1)/k_m}^{j/k_m}Z(x,y)~dx ~dy.$$
\end{itemize}
\end{lemma}

\addtocounter{section}{3}
\addtocounter{theorem}{-5}

The following lemma is an important step in the proof of Lemma~\ref{lemaLS.5.2}.

\begin{lemma}\label{rud.Zweak}
Let $(Q_m)_{m \in \mathbb{N}}$ be a sequence of weighted permutations $Q_m:[k_m]^2\to[0,1]$.
Assume that, for almost all $x,y\in[0,1]$, $Z_{Q_m}(x,y)$ converges as $m$ tends to infinity.
Then there exists a Lebesgue measurable function $Z:[0,1]^2\to[0,1]$ with the folowing properties:
\begin{itemize}
\item[(a)] for every $x\in[0,1]$, $Z(x,\cdot)$ is a cdf;
\item[(b)] $Z_{Q_m}(x,y)\to Z(x,y)$ for almost all $(x,y)\in[0,1]^2$;
\item[(c)] $Z_{Q_m}(x,\cdot)\overset{w}{\longrightarrow} Z(x,\cdot)$ for almost all $x\in[0,1]$.
\end{itemize}
\end{lemma}

\begin{proof} By hypothesis, for almost all $(x,y)$, the limit
$$
Z(x,y)=\lim_{m\to\infty}Z_{Q_m}(x,y)
$$
exists. As a consequence, for almost all $y\in[0,1]$, the sequence $(Z_{Q_m}(x,y))$ converges, as $m \to \infty$, for almost all $x\in[0,1]$.
In other words, there exists $\Delta\subset[0,1]$ with measure zero for which, if $y\in[0,1]\backslash\Delta$, then there is $\Gamma_y\subset[0,1]$ with measure zero such that $Z_{Q_m}(x,y)$ converges whenever $x\in[0,1]\backslash\Gamma_y$.

Note that the set $[0,1]\backslash\Delta$ is separable, since separability is hereditary for metric spaces and $[0,1]$ is separable. Let $\Psi$ be a countable set that is dense in $[0,1]\backslash\Delta$. Clearly,
$\Psi$ has measure zero, since it is countable, and is dense in $[0,1]$, since $\Delta$ has measure zero.

We consider the set $\Gamma_{\Psi}=\bigcup_{y\in\Psi}\Gamma_y$, which has measure zero for being the countable union of sets of measure zero. For every $x\in[0,1]\backslash\Gamma_{\Psi}$, we know that $Z_{Q_m}(x,y)$ converges to $Z(x,y)$ for every $y\in\Psi$.
Moreover, as the columns of $Q_m$ are non-decreasing, we also know that, for almost all $x\in[0,1]$,
the sequence of cdfs $(Z_{Q_m}(x,\cdot))_{m \in \mathbb{N}}$ converges pointwise in $\Psi$, which is dense in $[0,1]$. It is a fact that pointwise convergence of a cdf in a dense subset of a set implies weak convergence in the entire set, so that, for every $x\in[0,1]\backslash\Gamma_{\Psi}$, the sequence of cdfs $(Z_{Q_m}(x,\cdot))_{m \in \mathbb{N}}$ converges weakly to a function  $F_x(\cdot)$ as $m$ tends to infinity, and this limit is unique (for a proof of these facts, see~\cite{loeve77}, Chapter IV, \S 11.2).

For $x\in[0,1]\backslash \Gamma_\Psi$ and $y\in[0,1]$, we set $Z(x,y)=F_x(y)$. So far, for almost all $x\in[0,1]$ (recall that $\Gamma_{\Psi}$ has measure zero), we have
$$
Z_{Q_m}(x,\cdot)\overset{w}{\longrightarrow}Z(x,\cdot).
$$
Now, for $x\in\Gamma_{\Psi}$, set $Z(x,y)=y$, for every $y\in[0,1]$.
This fully defines a function $Z$ over $[0,1]^2$, which is Lebesgue measurable,
since it is obtained by pointwise convergence of the measurable functions $Z_{Q_m}$. It is clear by construction that the properties (a), (b) and (c) in the statement of the theorem are satisfied. \end{proof}

\begin{proof}[Proof of Lemma~\ref{lemaLS.5.2}] Lemma 5.3 in~\cite{lovasz06} ensures that, if $X$ and $Y$ are independent uniform random variables in $(0,1]$,
then the sequence $(Z_{Q_m}(X,Y))_{m \in \mathbb{N}}$ is a martingale. Because this martingale is bounded, the Martingale Convergence Theorem, see ~\cite{doob53}, Chapter VII, implies that $\lim_{m\to\infty}Z_{Q_m}(X,Y)$ exists with probability one. Equivalently,
$Z_{Q_m}(x,y)$ converges for almost all $(x,y)\in[0,1]^2$. By Lemma~\ref{rud.Zweak} in the present paper, there is a measurable function $Z:[0,1]^2\to[0,1]$ such that $Z(x,\cdot)$ is a cdf for every $x\in[0,1]$ for which conditions (i) and (ii) in the statement of Lemma~\ref{lemaLS.5.2} hold.

We now prove condition (iii). Observe that, because the sequence $(Q_j)_{j\in \mathbb{N}}$ satisfies property (i) in Lemma~\ref{lemaLS.5.1}, we have, for $i,j\in[k_m]$,
\begin{equation*}
\begin{split}
  Q_m(i,j) &= \lim_{n\to\infty}\Big(\frac{k_m}{k_n}\Big)^2
  \sum_{x=(i-1)(k_n/k_m)+1}^{i(k_n/k_m)} \sum_{y=(j-1)(k_n/k_m)+1}^{j(k_n/k_m)}Q_n(x,y)\\
&  =k_m^2\lim_{n\to\infty}\int_{(i-1)/k_m}^{i/k_m}\int_{(j-1)/k_m}^{j/k_m}Q_n(\phi_n(x),\phi_n(y))dx\ dy\\
& =k_m^2\lim_{n\to\infty}\int_{(i-1)/k_m}^{i/k_m}\int_{(j-1)/k_m}^{j/k_m}Z_{Q_n}(x,y)dx\ dy\\
&  =\ k_m^2\int_{(i-1)/k_m}^{i/k_m}\int_{(j-1)/k_m}^{j/k_m}Z(x,y)dx\ dy.
\end{split}
\end{equation*}
The last step follows by the Dominated Convergence Theorem, as $Z_{Q_n}$ is bounded and converges to $Z$ almost everywhere.

It remains to show that $Z$ is a limit permutation. We have already verified that, for every $x\in[0,1]$, $Z(x,\cdot)$ is a cdf. We now show that condition (b) in Definition~\ref{def.perm.lim} holds. To this end, suppose for a contradiction that $Z$ does not satisfy this condition. In other words, there exists $y\in[0,1]$ such that
\[
  \int_0^1 Z(x,y)\ dx\ =\ y+\Delta,\mbox{ for some }-y<\Delta<1-y,\mbox{ where }|\Delta|>0.
\]
First assume that $y\notin \{0,1\}$. Fix $k_m$ such that $k_m>\max\{12/|\Delta|,2/y,2/(1-y)\}$. By property (iii), given $j' \in [k_m]$, we have
\[
  \sum_{i=1}^{k_m} Q_m(i,j')
  =\sum_{i=1}^{k_m} k_m^2\int_{(i-1)/k_m}^{i/k_m}\int_{(j'-1)/k_m}^{j'/k_m}Z(x,y)dx\ dy
  =k_m^2\int_0^1\int_{(j'-1)/k_m}^{j'/k_m}Z(x,y)dx\ dy.
\]
Fix $j$ such that $\frac{j}{k_m}\leq y\leq\frac{j+1}{k_m}$. Now, using property (b) in the definition of weighted permutation, we deduce that
\begin{equation}\label{eq.LS5.2a}
  \frac{j-1}{k_m^2}\leq\ \int_0^1\int_{(j-1)/k_m}^{j/k_m}Z(x,y)dx\ dy\ \leq\frac{j}{k_m^2}
  \ \leq\ \int_0^1\int_{j/k_m}^{(j+1)/k_m}Z(x,y)dx\ dy\ \leq\frac{j+1}{k_m^2}.
\end{equation}
Since $Z(x,\cdot)$ is non-decreasing for every $x\in[0,1]$, it is easy to see that
\begin{equation}\label{eq.LS5.2b}
  \frac{1}{1/k_m}\int_{(j-1)/k_m}^{j/k_m} Z(x,y)dy\ \leq\ Z\Big(x,\frac{j}{k_m}\Big)\
  \leq\ \frac{1}{1/k_m}\int_{j/k_m}^{(j+1)/k_m} Z(x,y)dy.
\end{equation}
Combining equations~(\ref{eq.LS5.2a}) and~(\ref{eq.LS5.2b}), we have
\begin{equation}\label{eq.LS5.2c}
\frac{j-1}{k_m}\leq\ \int_0^1Z\Big(x,\frac{j}{k_m}\Big)dx\ \leq\frac{j+1}{k_m}.
\end{equation}

Since, by our choice of $j$ and $k_m$, we have $1\leq j \leq k_m y < k_m-2$, equation~(\ref{eq.LS5.2c}) may be rewritten as
\begin{equation}\label{eq.LS5.2d}
  \frac{j-1}{k_m}\leq\ \int_0^1Z\Big(x,\frac{j}{k_m}\Big)dx\ \leq\frac{j+1}{k_m}
  \leq\ \int_0^1Z\Big(x,\frac{j+2}{k_m}\Big)dx\ \leq\frac{j+3}{k_m}.
\end{equation}
Using this together with the fact that $Z(x,\cdot)$ is non-decreasing for every $x\in[0,1]$, as well as our choice of $j$, we have
\begin{equation*}
\begin{split}
 y-\frac{2}{k_m}&\leq
  \frac{j-1}{k_m}\leq\ \int_0^1Z\Big(x,\frac{j}{k_m}\Big)dx\ \leq \int_0^1 Z(x,y)\ dx\\
 & \leq\ \int_0^1Z\Big(x,\frac{j+2}{k_m}\Big)dx\ \leq\frac{j+3}{k_m}\leq y+\frac{3}{k_m}.
\end{split}
\end{equation*}
However, this leads to
\[
  y-\frac{|\Delta|}{4}\ <\ y-\frac{3}{k_m}\ \leq\
  \int_0^1 Z(x,y)\ dx\ \leq\ y+\frac{3}{k_m}\ <\ y+\frac{|\Delta|}{4},
\]
contradicting the fact that
$$\int_0^1 Z(x,y)\ dx\ =\ y+\Delta.$$
If $y=0$ or $y=1$, we may use the same type of argument, but we only need to bound the value of $\displaystyle{\int_0^1 Z(x,y)\ dx\ }$ from one side. This establishes part (b) in Definition~\ref{def.perm.lim}.

Furthermore, we may suppose that, for every $x \in [0,1]$, $Z(x,\cdot)$ is continuous at the point $y=0$ and $Z(x,1)=1$, otherwise we redefine $Z$: $Z(x,\cdot)$ is replaced by the uniform cdf $Z_u(y)=y$. Observe that this affects only a set of measure zero: for $y=0$, this is an immediate consequence of Lemma~\ref{Lema.DescNula}; for $y=1$, fix $\eps>0$ and consider the set $A_{\eps}=\{x \in [0,1]: Z(x,1) \leq 1-\eps\}$. By definition, it is clear that
$$\int_{0}^1 (1-Z(x,1))~dx \geq \eps \lambda(A_{\eps}),$$
which implies that $1=\int_{0}^1 Z(x,1)~dx \leq 1 - \eps \lambda(A_{\eps}) \leq 1$. Hence $\lambda(A_{\eps})=0$, and, as $\eps>0$ is arbitrary, we indeed have that $Z(x,1)=1$ for almost all $x \in [0,1]$.

Clearly, properties (i), (ii) and (iii) are not affected by this change.
\end{proof}

% --------------------------------------------------------------------
% --------------------------------------------------------------------
% --------------------------------------------------------------------
% --------------------------------------------------------------------
\subsection{The proof of Lemma~\ref{lem.QtoZ}} \label{provalem.QtoZ}
Recall that the objective of this section is to prove the continuity of subpermutation density under the limit of Lemma~\ref{lemaLS.5.2}.

\addtocounter{section}{-3}
\addtocounter{theorem}{4}

\begin{lemma}
Let $(k_m)_{m \in \mathbb{N}}$ be a sequence of positive integers and let $(Q_m)_{m \in \mathbb{N}}$ be weighted permutations with the properties of Lemma~\ref{lemaLS.5.2}.
Then the function $Z$ given by Lemma~\ref{lemaLS.5.2} satisfies
$$\lim_{n \to \infty} t(\tau,Q_n)=t(\tau,Z) \textrm{ for every permutation }\tau.$$
\end{lemma}

\addtocounter{section}{3}
\addtocounter{theorem}{-5}

The proof of this result relies on the following auxiliary lemma, which depends on measure-theoretical machinery such as Alexandrov's Portmanteau Theorem and the Multivariate Helly--Bray Theorem, whose statetements are included below (a reference for both theorems is~\cite{loeve77}, Chapter IV).

\begin{lemma}\label{app.portmanteau}
Let $Z$ be a limit permutation. Consider a sequence of weighted permutations $(Q_n)_{n \in \mathbb{N}}$, $Q_n:[k_n]^2\to[0,1]$, such that
$Z_{Q_n}(x,y)\to Z(x,y)$ for almost all $x,y\in[0,1]$ and
$Z_{Q_n}(x,\cdot)\overset{w}{\longrightarrow}Z(x,\cdot)$ for almost all $x\in[0,1]$.
For $m>1$, let $\tau:[m]\to[m]$ be a permutation.
Then, for almost all $x\in[0,1]^m$, we have
\begin{equation}\label{eq.[0,1]good}
  \lim_{n\to\infty}L_{\tau,Z_{Q_n}}(x)\ =\ L_{\tau,Z}(x).
\end{equation}
\end{lemma}

\begin{theorem}[Alexandrov's Portmanteau Theorem]\label{alexandrov}
Let $\Omega$ be a separable metric space and let $\mu, \mu_1, \mu_2, \ldots$ be probability measures over the Borel sets of $\Omega$. The following assertions are equivalent.
\begin{itemize}
\item[(a)] For every bounded continuous function $g:\Omega \to \mathbb{R}$, $\lim_{n \to \infty} \int_{\Omega} g~d\mu_n \, =  \int_{\Omega} g~d\mu.$  

\item[(b)] $\lim_{n \to \infty} \mu_n(A)=\mu(A)$ for every $A \subseteq \Omega$ whose boundary $\partial A$ satisfies $\mu(\partial A)=0$. The \emph{boundary} $\partial A$ of $A$ is the set of points $x$ of $A$ for which every open ball centered at $x$ contains a point in $\Omega\backslash A$.

\item[(c)] $\limsup_{n \to \infty} \mu_n(C) \leq \mu(C)$ for every closed set $C \subseteq \Omega$.

\item[(d)] $\liminf_{n \to \infty} \mu_n(U) \geq \mu(U)$ for every open set $U \subseteq \Omega$.
\end{itemize}
\end{theorem}

\begin{theorem}[Multivariate Helly--Bray Theorem]\label{helly}
Let $k>1$ be an integer and let $g:[0,1]^k \to \mathbb{R}$ be a bounded continuous function. For every $i \in [k]$, let $F_i$ be a cdf and $(F_{i,n})_{n \in \mathbb{N}}$ be a sequence of cdf such that $F_{i,n}\overset{w}{\longrightarrow} F_i$. Then
$$\lim_{n \to \infty} \int_{[0,1]^k} g~dF_{1,n}\cdots dF_{k,n} =  \int_{[0,1]^k} g~dF_{1}\cdots dF_{k}.$$ 
\end{theorem}

\begin{proof}[Proof of Lemma~\ref{app.portmanteau}] Let $[0,1]^m_w\subset[0,1]^m$ be the subset consisting of the points
$x=(x_1,\ldots,x_m)$ such that, as $n$ tends to infinity,
$Z_{Q_n}(x_i,\cdot)\overset{w}{\longrightarrow}Z(x_i,\cdot)$ for every $i\in[m]$. Our hypothesis implies that the measure of  
$[0,1]^m\backslash[0,1]^m_w$ is zero.

Let $[0,1]^m_d$ be the set of points $(x_1,\ldots,x_m)$ of $[0,1]^m$ such that, for some pair $1\leq i<j\leq m$,
$Z(x_i,\cdot)$ and $Z(x_j,\cdot)$ have a common point of discontinuity.
By Lemma~\ref{lema.critico}, we know that $[0,1]^m_d$ has measure zero.
Let $[0,1]^m_{good}$ be given by $[0,1]^m_{good}=[0,1]^m_w\backslash[0,1]^m_d$, whose complement with respect to $[0,1]$ clearly has measure zero.

We prove that every $x\in[0,1]^m_{good}$ satisfies equation~(\ref{eq.[0,1]good}). To this end, fix $x=(x_1,\ldots,x_m)\in[0,1]^m_{good}$ and let $\mu$ be the measure in $[0,1]^m$ given by the product of the Lebesgue-Stieltjes measures associated with the cdf $Z(x_{\tau^{-1}(i)},\cdot)$, for each $i\in[m]$. For all $n>1$, let $\mu_n$ be the measure in $[0,1]^m$ given by the product of the Lebesgue-Stieltjes measures associated with the cdf $Z_{Q_n}(x_{\tau^{-1}(i)},\cdot)$, for each $i\in[m]$.
Since $x\in[0,1]^m_w$, the Multivariate Helly--Bray Theorem ensures that $\mu_n$ converges weakly to $\mu$ as $n$ tends to infinity.

To prove that $\lim_{n\to\infty}\mu_n([0,1]^m_<)=\mu([0,1]^m_<)$, which is precisely equation~(\ref{eq.[0,1]good}),
it suffices to apply Alexandrov's Portmanteau Theorem. However, we still need to verify that $\mu(\partial[0,1]^m_<)=0$. Note that the boundary $\partial[0,1]^m_<$ has Lebesgue measure zero in $[0,1]^m$, as it consists of the points $(\alpha_1,\ldots,\alpha_m)\in[0,1]^m$ such that $\alpha_1 \leq \cdots \leq \alpha_m$ and $\alpha_i=\alpha_{i+1}$ for some $i\in[m-1]$. For $i\in[m-1]$, let $\Omega_i\subset[0,1]^m$ be the hyperplane of the points $(\alpha_1,\alpha_2,\ldots,\alpha_m)$ such that $\alpha_{i+1}=\alpha_i$.
Clearly, $\partial[0,1]^m_<\subset\bigcup_{i=1}^{m-1}\Omega_i$, so that $\mu(\partial[0,1]^m_<)\leq\sum_{i=1}^{m-1}\mu(\Omega_i)$.

By Fubini's Theorem, we have
$$\mu(\Omega_1)\ =\ \int_{\alpha_1\in[0,1]}\int_{\alpha_2\in\{\alpha_1\}}\int_{\alpha_3\in[0,1]}\int_{\alpha_4\in[0,1]}
  \cdots\int_{\alpha_m\in[0,1]}dZ(x_{\tau^{-1}(1)},\cdot)\cdots dZ(x_{\tau^{-1}(m)},\cdot).$$
Now, given that, for every $k\in[m]$, $\displaystyle{\int_{\alpha_k\in[0,1]}dZ(x_{\tau^{-1}(k)},\cdot)=1}$, we have
$$\mu(\Omega_1)=\int_{\alpha_1\in[0,1]}\Big(\int_{\alpha_2\in\{\alpha_1\}}
  dZ(x_{\tau^{-1}(2)},\cdot)\Big) dZ(x_{\tau^{-1}(1)},\cdot).
$$
Let $D_2$ be the set of discontinuities of the cdf $Z(x_{\tau^{-1}(2)},\cdot)$. The internal integral is positive when $\alpha_1\in D_2$ and is equal to zero when this is not the case. As a consequence,
$$
\mu(\Omega_1)\leq\int_{\alpha_1\in D_2} dZ(x_{\tau^{-1}(1)},\cdot)=\mu'_{\tau^{-1}(1)}(D_2),
$$
where $\mu'_{\tau^{-1}(1)}$ is the Lebesgue-Stieltjes measure associated with the cdf $Z(x_{\tau^{-1}(1)},\cdot)$.
Now, it is an easy fact that the set $D_2$ is countable (see, for instance,~\cite{loeve77}). As a consequence, the countable-additivity of the measure $\mu'_{\tau^{-1}(1)}$ leads to $\mu'_{\tau^{-1}(1)}(D_2)=\sum_{\alpha\in D_2}\mu'_{\tau^{-1}(1)}(\alpha)$.

Since $x\not\in[0,1]^m_d$, we know that the cdfs $Z(x_{\tau^{-1}(1)},\cdot)$ and $Z(x_{\tau^{-1}(2)},\cdot)$ cannot have a common discontinuity, so that $\mu_{\tau^{-1}(1)}(\alpha)=0$ for every $\alpha\in D_2$.
Hence, $\mu'_{\tau^{-1}(1)}(D_2)=0$ and $\mu(\Omega_1)=0$. Analogously, we have $\mu(\Omega_i)=0$ for every $i\in[m-1]$, so that $\mu(\partial[0,1]^m_<)=0$ and our result follows from the Portmanteau Theorem. \end{proof}

Before proceeding with the proof of Lemma~\ref{lem.QtoZ}, we first use the above result to establish that, for a limit permutation $Z$ and a permutation $\tau$, the core function $L_{\tau,Z}$ is Lebesgue measurable, and hence the density of $\tau$ as a subpermutation of the limit permutation $Z$ is well defined, see Definition~\ref{def.dens.subperm.lim}.

\begin{proposition}\label{propQtoZ}
Let $Z$ be a limit permutation and let $\tau: [m] \to [m]$ be a permutation, where $m>1$. Then the core function $L_{\tau,Z}$ is Lebesgue measurable.
\end{proposition}

\begin{proof}
We use a strategy that resembles the work of Borgs \etal~\cite{borgs06b}. Let $P$ be a partition of $[0,1]$ in $k$ intervals $V_1=[a_1,a_2]$ and $V_i=(a_{i},a_{i+1}]$, $i=2,\ldots,k$, where $a_1=0$ and $a_{k+1}=1$. For every $i \in [k]$, let $|V_i|=a_{i+1}-a_i$.

We consider the matrix $Q_{Z,P}:[k]^2 \to [0,1]$, where, for every $i,j \in [k]$,
$$Q_{Z,P}(i,j)=\frac{1}{|V_i| |V_j|}\int_{V_i \times V_j} Z(x,y) dx \, dy.$$

For every $n \geq 1$, let $P_n$ be a partition of $[0,1]$ in intervals of length $1/2^n$, and let $Q_n=Q_{Z,P_n}$. Clearly, for $n'>n>0$, the partition $P_{n'}$ refines $P_n$. As in the proof of Lemma~\ref{lemaLS.5.2}, we may show through a martingale that, as $n$ tends to infinity, $Z_{Q_n}$ converges almost everywhere and that the limit is $Z$. Lemma~\ref{lemaLS.5.2} further implies that, for almost all $x \in [0,1]$, $Z_{Q_n}(x,\cdot) \overset{w}{\longrightarrow} Z(x,\cdot)$.

On the one hand, by Lemma~\ref{app.portmanteau}, the core function $L_{\tau,Z}$ is the limit of the sequence of functions $(L_{\tau,Z_{Q_n}})_{n \in \mathbb{N}}$ almost everywhere. On the other hand, it is easy to see that all the functions $L_{\tau,Z_{Q_n}}$ are measurable, as $Z_{Q_{n}}$ is a step function for every $n$.

As $L_{\tau,Z}$ is the limit of a sequence of Lebesgue measurable functions on $[0,1]$, it is itself measurable.
\end{proof}

\begin{proof}[Proof of Lemma~\ref{lem.QtoZ}] Let $\tau$ be a permutation on $[m]$. Let $L_n$ denote the core function $L_{\tau,Z_{Q_n}}$. As $Z_{Q_n}$ is a step function, it is easy to see that, for every $x=(x_1,\ldots,x_m)\in [0,1]^m$,
\begin{equation}\label{tapsi1}
  L_n(x)\ =\ \sum_{A\in[k_n+1]^m_<}\prod_{i=1}^m
  \Big(Q_n(\lceil k_nx_i\rceil,a_{\tau(i)})-Q_n(\lceil k_nx_i\rceil,a_{\tau(i)}-1)\Big).
\end{equation}
Definition~\ref{def.dens.perm2} states that
\[
  t(\tau,Q_n)\ =\ \binom{k_n}{m}^{-1}\sum_{X\in[k_n]^m_<}\sum_{A\in[k_n+1]^m_<}\prod_{i=1}^m
  \Big(Q_n(X_i,a_{\tau(i)})-Q_n(X_i,a_{\tau(i)}-1)\Big),
\]
which, using~(\ref{tapsi1}), may be rewritten as
\[
  t(\tau,Q_n)\ =\ \binom{k_n}{m}^{-1}\sum_{X\in[k_n]^m_<}L_n\Big(\frac{X_1}{k_n},\ldots,\frac{X_m}{k_n}\Big).
\]
Moreover, note that $\displaystyle{L_n\big(x_1,\ldots,x_m\big)=L_n\big(X_1/k_n,\ldots, X_m/k_n\big)}$
whenever $x_i$ lies in the interval $\big((X_i-1)/k_n,X_i/k_n\big]$ for every $i\in[m]$.
As a consequence, the summation $\frac{1}{k_n^m}\sum_{X\in[k_n]^m_<}$ may be rewritten as a multiple integral, where $x_1$ ranges from 0 to 1, $x_2$ ranges from $\frac{1}{k_n}\lceil k_nx_1\rceil$ to 1, and, in general,
$x_j$ ranges from $\frac{1}{k_n}\lceil k_nx_{j-1}\rceil$ to 1. Thus
\begin{equation}\label{tapsi2}
  t(\tau,Q_n)=\binom{k_n}{m}^{-1} k_n^m
  \int_{(0,1]}\int_{\Big(\frac{1}{k_n}\lceil k_nx_1\rceil,1\Big]}
  \cdots\int_{\Big(\frac{1}{k_n}\lceil k_nx_{m-1}\rceil,1\Big]}
  L_n(x_1,\ldots,x_m)\ dx_1\cdots dx_m.
\end{equation}
Since $L_n$ is a nonnegative function, this immediately leads to the upper bound
\[
  t(\tau,Q_n)\ \leq\ \binom{k_n}{m}^{-1}k_n^m
  \int_{[0,1]^m_<} L_n(x)\ dx.
\]
We claim that the following lower bound holds.
\[
  t(\tau,Q_n) \geq\ \binom{k_n}{m}^{-1}k_n^m \Big(
  \int_{[0,1]^m_<} L_n(x)\ dx\  -\ \frac{m-1}{k_n}\Big).
\]
Indeed, all the points $(x_1,\ldots,x_m)$ that lie in $[0,1]^m_<$, but are not in the region of integration of equation~(\ref{tapsi2}), are such that, for some $i \in [m-1]$, $x_{i+1} \in (x_i,\frac{1}{k_n}\lceil k_nx_i\rceil)$. As a consequence,
\[
  t(\tau,Q_n)\geq\binom{k_n}{m}^{-1}k_n^m\Big(\int_{[0,1]^m_<}L_n(x)dx -
  \sum_{i=1}^{m-1}\int_0^1\cdots\int_{x_{i-1}}^1\int_{x_i}^{\frac{1}{k_n}\lceil k_nx_i\rceil}
  \int_{x_{i+1}}^1\cdots\int_{x_{m-1}}^1 L_n(x)dx\Big)
\]
\[
  \geq\binom{k_n}{m}^{-1}k_n^m\Big(\int_{[0,1]^m_<}L_n(x_1,\ldots,x_m)
  \ dx_1\cdots dx_m\ -\  \sum_{i=1}^{m-1}\frac{1}{k_n}\Big),
\]
as claimed. Because $k_n$ goes to infinity with $n$, our upper and lower bounds lead to
\[
  \lim_{n\to\infty}t(\tau,Q_n)\ =\
  \lim_{n\to\infty}\binom{k_n}{m}^{-1}k_n^m\int_{[0,1]^m_<} L_n(x)\ dx.
\]
Now, we use the definition of $L_n$ together with the fact that $\binom{k_n}{m}^{-1}k_n^m\to m!$ to obtain
\[
  \lim_{n\to\infty}t(\tau,Q_n)\ =\ m!\lim_{n\to\infty}\int_{[0,1]^m_<}\Big(\int_{[0,1]^m_<}
    dZ_{Q_n}(x_{\tau^{-1}(1)},\cdot)\cdots dZ_{Q_n}(x_{\tau^{-1}(m)},\cdot)\Big) dx_1\cdots dx_m.
\]
By Lemma~\ref{app.portmanteau}, there is a set $[0,1]^m_{good,<}\subset[0,1]^m_<$
such that the set $[0,1]^m_<\backslash[0,1]^m_{good,<}$ has measure zero, and,
for every $x=(x_1,\ldots,x_m)\in[0,1]^m_{good,<}$, equation~(\ref{eq.[0,1]good}) holds. Note that Lemma~\ref{app.portmanteau} may be applied because we assume that the hypotheses of Lemma~\ref{lemaLS.5.2} hold.

To conclude the proof, we combine the fact that $[0,1]^m_<\backslash[0,1]^m_{good,<}$ has measure zero with an application of the Dominated Convergence Theorem over the set $[0,1]^m_{good,<}$. This yields
\begin{equation*}
\begin{split}
  \lim_{n\to\infty}t(\tau,Q_n)\ &=\ m!\int_{[0,1]^m_{good,<}}\Big(\lim_{n\to\infty}\int_{[0,1]^m_<}
    dZ_{Q_n}(x_{\tau^{-1}(1)},\cdot)\cdots dZ_{Q_n}(x_{\tau^{-1}(m)},\cdot)\Big) dx_1\cdots dx_m\\
 &=\ m!\int_{[0,1]^m_<}\Big(\int_{[0,1]^m_<}  
    dZ(x_{\tau^{-1}(1)},\cdot)\cdots dZ(x_{\tau^{-1}(m)},\cdot) \big) dx_1\cdots dx_m\ =\ t(\tau,Z),
\end{split}
\end{equation*}
as required. \end{proof}

% --------------------------------------------------------------------
% --------------------------------------------------------------------
% --------------------------------------------------------------------
% --------------------------------------------------------------------

\subsection{The proof of Lemma~\ref{lemaLS.5.1}}\label{provalemaLS.5.1}

This section is devoted to the proof of Lemma~\ref{lemaLS.5.1}, which draws its inspiration from the proof of Lemma 5.1 in~\cite{lovasz06}, with graphs and the related notions of rectangular distance and weak regularity being replaced by their permutation counterparts. We remind the reader of the statement of this result.

\addtocounter{section}{-3}
\addtocounter{theorem}{-4}

\begin{lemma}
Every permutation sequence $(\sigma_n)_{n \in \mathbb{N}}$ with $\lim_{n \to \infty}|\sigma_n|=\infty$ has a subsequence $(\sigma'_m)_{m \in \mathbb{N}}$, $|\sigma_m'| \geq m$, for which there exist a sequence of positive integers $(k_m)_{m \in \mathbb{N}}$ and a sequence of weighted permutations $(Q_m)_{m \in \mathbb{N}}$, $Q_m:[k_m]^2\to[0,1]$, satisfying the following properties.
\begin{itemize}
\item[(i)] If $j<m$, then $k_j$ divides $k_m$ and $Q_j=\widehat{Q}_{m,j}$, where
$\widehat{Q}_{m,j}$ is the matrix obtained from $Q_m$ by merging its entries in $k_j$ consecutive blocks of size $k_m/k_j$, and by replacing each block by a single value, the arithmetic mean of the entries in that block.
\item[(ii)] For every $j<m$, $\sigma'_m$ has a weak $(1/j)$-regular $k_j$-partition $P_{m,j}$ whose partition matrix $Q_{m,j}$ has dimension $k_j\times k_j$ and satisfies
$$d_{\square}(Q_{m,j},Q_j)<1/j.$$
Moreover, for $1\leq i<j\leq m$, $P_{m,j}$ refines $P_{m,i}$.  
\end{itemize}
\end{lemma}

\addtocounter{section}{3}
\addtocounter{theorem}{3}

\begin{proof} For every integer $m \geq 1$, we shall define a subsequence $(\sigma_{n}^m)_{n \in \mathbb{N}}$ of $(\sigma_{n})_{n\in \mathbb{N}}$, as well a positive integer $k_m$ and a weighted permutation $Q_m$, for which the following properties hold:
\begin{itemize}
\item[(i)] for $m'>m$, $(\sigma_{n}^{m'})$ is a subsequence of $(\sigma_n^m)$;

\item[(ii)] for every pair $(m,n) \in \mathbb{N}^2$, $|\sigma_n^m| \geq m$;

\item[(iii)] for every pair $(m,n) \in \mathbb{N}^2$, $\sigma_{n}^m$ has a weak $1/m$-regular equitable $k_m$-partition $P_n^m$ that, if $m>1$, refines the partition associated with $\sigma_n^m$ as an element of the sequence $\sigma_n^{m-1}$. Moreover, if $Q_n^m$ denotes the partition matrix of $\sigma_n^m$ induced by $P_n^m$, we have $d_{\square}(Q_{n}^m,Q_m) < 1/m$.

\item[(iv)] if $j<m$, then $k_j$ divides $k_m$ and $Q_j=\widehat{Q}_{m,j}$, where
$\widehat{Q}_{m,j}$ is the matrix obtained from $Q_m$ by merging its entries in $k_j$ consecutive blocks of size $k_m/k_j$, and by replacing each block by a single value, the arithmetic mean of the entries in that block.
\end{itemize}
It is not difficult to see that any sequence $(\sigma_m')_{m \in \mathbb{N}}$ whose $m$-th element lies in $(\sigma_n^m)_{n \in \mathbb{N}}$ satisfies the requirements of the lemma.

We show by induction on $m$ that such a construction may be obtained. Set $(\sigma_{n}^1)_{n \in \mathbb{N}}$ to be the sequence $(\sigma_n)_{n \in \mathbb{N}}$, let $k_1=1$ and $Q_1=[1/2]$. For each permutation $\sigma^1_{n}$, let $P_{n}^1$ be the partition of $[|\sigma^1_{n}|]$ into a single interval. Clearly, for every $n$, the density matrix of $\sigma_{n}^{1}$ with respect to $P_{n}^1$ is given by $Q_{n}^1=[(n-1)/2n]$. All the required conditions are trivially satisfied.

Now, inductively assume that, for some $m>1$, we have $(\sigma_n^j)_{n \in \mathbb{N}}$, $k_j$, $Q_j$ and $(P_n^j)$ for $j \in \{1,\ldots,m-1\}$. Using Lemma~\ref{prop_weak}, fix $k_0^m=k_0(1/m)$ and define
$$k_m=\min \left\{i k_{m-1}: i \in \mathbb{N} \textrm{ and } i k_{m-1}> k_0^m\right\}.$$
Consider the subsequence $(\hat{\sigma}_n^m)_{n \in \mathbb{N}}$ of $(\sigma_{n}^{m-1})_{n \in \mathbb{N}}$ consisting of all $\sigma$ with $|\sigma|>\max\{m,4 k_m^2\}$.

Now, for every element $\hat{\sigma}_n^m$, consider a $k_m$-partition $P_n^m$ whose elements are $(k_m/k_{m-1})$-equitable partitions of each interval in the partition corresponding to this permutation in step $m-1$, which we call $P_{n}^{m-1}$ for simplicity. We claim that each $P_n^m$ is an equitable partition of $[|\hat{\sigma}_n^m|]$. Indeed, for $c=\left\lfloor |\hat{\sigma}_n^m|/k_{m-1}\right\rfloor$, we know that $c\leq|C_i|\leq c+1$ for every interval $C_i$ in $P_n^{m-1}$, $i\in[|\hat{\sigma}_n^m|]$, as $P_n^{m-1}$ is equitable. As a consequence, the intervals in $P_n^m$ have sizes ranging from $\left\lfloor k_{m-1} c / k_m\right\rfloor$ to $\left\lceil k_{m-1} (c+1) / k_m\right\rceil$. Now, suppose for a contradiction that $P_n^m$ is not equitable, so that, for some integer $i$, $\left\lceil  k_{m-1} (c+1) / k_m\right\rceil=i+1$ and $\left\lfloor  k_{m-1}c / k_m\right\rfloor=i-1$. By definition,
\[
  i-1 \leq k_{m-1}c/k_m < i <k_{m-1} (c+1) /k_m  \leq i+1.
\]
This implies that $c<i k_m/k_{m-1}<c+1$, contradicting the fact that $c$, $i$ and $k_m/k_{m-1}$ are integers. Hence, $P_n^m$ is equitable.

Note that, by Lemma~\ref{prop_weak}, our choice of $k_m$ implies that the partitions $P_n^m$ are weakly $1/m$-regular, while, by Lemma~\ref{prop_weighted}, our choice of $n$ tells us that the partition matrices $Q_{n}^m$ of the permutation $\hat{\sigma}_n^m$ induced by $P_n^m$, respectively, are weighted permutations.

Moreover, for every $(i,j) \in [k_m]^2$, the sequence of real numbers $(Q_n^m(i,j))_{n \in \mathbb{N}}$ is bounded. In particular, we may define a subsequence $(\sigma_n^m)_{n \in \mathbb{N}}$ of $(\hat{\sigma}_n^m)_{n \in \mathbb{N}}$ for which $Q_{n}^m(i,j)$ converges to a real number $Q_m(i,j)$ for every $(i,j) \in [k_m]^2$. Note that the matrix $Q_m$ is a weighted permutation, as it is the limit of weighted permutations. It is also easy to see that condition (iv) must be satisfied. Our result now follows if we restrict $(\sigma_n^m)$ to those elements for which $|Q_{n}^m(i,j)-Q_m(i,j)| \leq \frac{1}{m k_m^2}$, as we have $d_{\square}(Q_n^m,Q_m)<1/m$ for every $n$ in this case.
\end{proof}

% --------------------------------------------------------------------
% --------------------------------------------------------------------
% --------------------------------------------------------------------
% --------------------------------------------------------------------

\noindent \emph{Acknowledgements:} The authors are indebted to Bal\'{a}zs R\'{a}th and to an anonymous referee for valuable comments and suggestions. The authors are also grateful to NUMEC/USP, N\'{u}cleo de Modelagem Estoc\'{a}stica e Complexidade of the University of S\~{a}o Paulo, for its hospitality.

\appendix

\section{Proofs of the results of Section~\ref{section_regularity}}\label{apend_A}

For completeness, this section contains the proofs of the results presented in Section~\ref{section_regularity}. \footnote{The proofs of these results, along with other properties of weighted permutations, can be found in the reference~\cite{rudini08b}.}

\begin{proof}[Proof of Lemma~\ref{prop_weighted}] Property (a) in the definition of weighted permutations is an immediate consequence of the definition of the graph $G_{\sigma}$. To prove property (b), let $j \in [k]$ and consider the expression
\[
  \sum_{i=1}^kQ_P(i,j)=\sum_{i=1}^k\frac{e_{\sigma}(C_i,C_j)}{|C_i||C_j|}
  =\frac{1}{|C_j|}\sum_{i=1}^k\frac{e_{\sigma}(C_i,C_j)}{|C_i|}.
\]
Let $a_j$ be the smallest element in $C_j$. It is clear that
\[
  (j-1)\Big\lfloor\frac{n}{k}\Big\rfloor+1\ \leq a_j\ \leq (j-1)\Big\lceil\frac{n}{k}\Big\rceil+1,
\]
\[
  (j-1)\Big(\Big\lceil\frac{n}{k}\Big\rceil-1\Big)+1\ \leq a_j\ \leq (j-1)\Big(\Big\lfloor\frac{n}{k}\Big\rfloor+1\Big)+1.
\]
Moreover, we have
\[
  \sum_{i=1}^k e_{\sigma}(C_i,C_j)=e_{\sigma}([n],C_j)=\sum_{h=a_j}^{a_j+|C_j|-1}(h-1)=\frac{|C_j|}{2}(a_j-1+a_j+|C_j|-2)
  =|C_j|\Big(a_j+\frac{|C_j|-3}{2}\Big).
\]
Combining the previous equations with the condition $j\leq k\leq n/4k$, we obtain
\[
  \sum_{i=1}^k e_{\sigma}(C_i,C_j)\leq |C_j|\Big\lfloor\frac{n}{k}\Big\rfloor
  \Big(j-1+\frac{j}{\lfloor n/k\rfloor}+\frac{|C_j|-3}{2\lfloor n/k\rfloor}\Big)\leq |C_j|\Big\lfloor\frac{n}{k}\Big\rfloor j,  
\]
\[
  \sum_{i=1}^k e_{\sigma}(C_i,C_j)\geq |C_j|\Big\lceil\frac{n}{k}\Big\rceil
  \Big(j-1-\frac{(j-2)}{\lceil n/k\rceil}+\frac{|C_j|-3}{2\lceil n/k\rceil}\Big)
  \geq |C_j|\Big\lceil\frac{n}{k}\Big\rceil (j-1).
\]
Finally, these upper and lower bounds lead to
\[
  \sum_{i=1}^k Q_P(i,j)= \frac{1}{|C_j|}\sum_{i=1}^k\frac{e_{\sigma}(C_i,C_j)}{|C_i|}
  \leq \frac{1}{|C_j|\lfloor\frac{n}{k}\rfloor}\sum_{i=1}^k e_{\sigma}(C_i,C_j)
  \leq \frac{1}{|C_j|\lfloor\frac{n}{k}\rfloor}|C_j|\Big\lfloor\frac{n}{k}\Big\rfloor j\ =\ j,
\]
\[
  \sum_{i=1}^k Q_P(i,j)= \frac{1}{|C_j|}\sum_{i=1}^k\frac{e_{\sigma}(C_i,C_j)}{|C_i|}
  \geq \frac{1}{|C_j|\lceil\frac{n}{k}\rceil}\sum_{i=1}^k e_{\sigma}(C_i,C_j)
  \geq \frac{1}{|C_j|\lceil\frac{n}{k}\rceil}|C_j|\Big\lceil\frac{n}{k}\Big\rceil (j-1)\ =\ j-1.
\]
This establishes our result.
\end{proof}

\begin{proof}[Proof of Lemma~\ref{prop_weak}]
Without loss of generality, we fix $0<\eps<1/2$ and let $k \geq 8/\eps^2$. Let $P$ be a uniform equitable $k$-partition of the set $[n]$, $n>2k$, and consider a permutation $\sigma$ on $[n]$. Let $Q$ be the partition matrix of $\sigma$ induced by $P$. The adjacency matrix $Q_{\sigma}$ and the blow-up matrix $\mathcal{K}(P,Q)$ are denoted by $Q_1$ and $Q_2$, respectively. Given subsets $S,T \subseteq I[n]$, we let $e(S,T)=e_\sigma(S,T)$ be the number of edges with one endpoint in $S$ and the other in $T$ in the permutation graph $G_\sigma$, while $d(S,T)=e(S,T)/(|S| |T|)$ is the edge density between these sets in $G_\sigma$. 

We need to show that $d_{\square}(Q_1,Q_2) \leq \eps$, that is, for every pair $a<b \in [n+1]$ and every interval $S \in I[n]$, the following inequality holds:
$$\Delta(a,b,S)=\frac{1}{n}
  \Big|\sum_{x\in S}\Big((Q_1(x,b)-Q_1(x,a))-(Q_2(x,b)-Q_2(x,a))\Big)\Big| \leq \eps.$$

So, fix $a<b \in [n+1]$ and an interval $S \in I[n]$. If $|S| \leq \eps n$, then the inequalities $0 \leq Q_{\ell}(x,b)-Q_{\ell}(x,a) \leq 1$ valid for $\ell \in \{1,2\}$ immediately lead to $\alpha(a,b,S) \leq \eps$.

We thus assume that $|S| > \eps n$. By definition, $Q_1(x,a)=1$ if $\sigma(x)<a$ and $Q_1(x,a)=0$ otherwise. 
\begin{equation}\label{tapsi1a}
\sum_{x \in S} Q_1(x,a)=e(S,\{a\})=|S|\cdot d(S,\{a\}).
\end{equation}
Clearly, we also have $\sum_{x \in S} Q_1(x,b)= d(S,\{b\}).$ Now, let $C_\alpha$ and $C_\beta$, $\alpha,\beta \in [k]$ be the intervals in the $k$-partition $P$ containing $a$ and $b$, respectively, and let $\ell,r \in [k]$ be the indices of the leftmost and the rightmost intervals of $P$ whose intersection with $S$ is non-empty. Then
\begin{equation}\label{tapsi2a}
\begin{split}
\sum_{x\in S} & Q_2(x,a)=\sum_{i=1}^{k} \sum_{x\in C_i \cap S} Q_2(x,a)= \sum_{i=1}^{k}\sum_{x \in C_i \cap S} d(C_i,C_{\alpha})=\sum_{i=1}^{k} |C_i \cap S| d(C_i,C_{\alpha})\\
&=\sum_{i=1}^{k} |C_i \cap S| d(C_i \cap S,C_{\alpha})+\sum_{q \in \{\ell,r\}}|C_q \cap S| \left(d(C_q,C_{\alpha})-d(C_q \cap S, C_{\alpha}) \right)\\
&= |S| d(S,C_{\alpha})+\sum_{q \in \{\ell,r\}}|C_q \cap S| \left(d(C_q,C_{\alpha})-d(C_q \cap S, C_{\alpha}) \right).
\end{split}
\end{equation}

The following analogous formula holds for $b$:
$$\sum_{x\in S} Q_2(x,b)= |S| d(S,C_{\beta})+\sum_{q \in \{\ell,r\}}|C_q \cap S| \left(d(C_q,C_{\beta})-d(C_q \cap S, C_{\beta}) \right).$$
As $|C_q \cap S| \leq |C_q| \leq \lceil n/k \rceil$, we have
$$\sum_{q \in \{\ell,r\}}|C_q \cap S| \left|d(C_q,C_{j})-d(C_q \cap S, C_{j}) \right| \leq 2 \left\lceil \frac{n}{k} \right\rceil \leq \frac{4n}{k}$$
for $j \in \{\alpha,\beta\}$. Using the triangle inequality and equations~(\ref{tapsi1a}),~(\ref{tapsi2a}) and their counterparts with $b$ replacing $a$, we have
\begin{equation}\label{tapsi3a}
\begin{split}
&\frac{1}{n}\Big|\sum_{x\in S}\Big((Q_1(x,b)-Q_1(x,a))-(Q_2(x,b)-Q_2(x,a))\Big)\Big|\\
&\leq |d(S,\{b\})-d(S,C_{\beta})|+|d(S,\{a\})-d(S,C_{\alpha})|+\frac{8}{k}.
\end{split}
\end{equation}

To conclude the proof, we estimate $|d(S,\{a\})-d(S,C_{\alpha})|$. Let $n_\ell$, $n_c$ and $n_r$ be the number of elements of $\sigma(S)$ on the left of $C_{\alpha}$, on the left of $a$ in $C_{\alpha}$ and on the right of $a$ in $C_{\alpha}$, respectively. By definition,
\begin{equation*}
\begin{split}
&n_{\ell}+n_c \leq e(S,\{a\}) \leq n_{\ell}+n_c +1,\\
&n_{\ell}|C_{\alpha}| \leq e(S,C_{\alpha}) \leq  (n_{\ell}+n_c +n_{r}+1) |C_{\alpha}|.
\end{split}
\end{equation*}
If we divide the first equation by $|S|$ and the second by $|S||C_{\alpha}|$, we may use them to obtain
$$-(n_r+1)/|S| \leq d(S,\{a\})-d(S,C_{\alpha}) \leq (n_c+1)/|S|.$$
As $n_c$ and $n_r$ are smaller than $|C_{\alpha}|<2n/k$, $k>8/\eps^2$ and $|S| \geq \eps n$, we have
$$|d(S,\{a\})-d(S,C_{\alpha})| \leq |C_{\alpha}|/|S| \leq \frac{2}{k\eps} \leq \frac{\eps}{4}.$$
Analogously, $|d(S,\{b\})-d(S,C_{\beta})|\leq \eps/4$. Equation~(\ref{tapsi3a}) now tells us that
$$\frac{1}{n}\Big|\sum_{x\in S}\Big((Q_1(x,b)-Q_1(x,a))-(Q_2(x,b)-Q_2(x,a))\Big)\Big| \leq \eps/2 + \eps^2 < \eps,$$
as $\eps<1/2$. This establishes our result.
\end{proof}

\begin{proof}[Proof of Lemma~\ref{lema.cut.subpermut}] Let $m=|\tau|$ and define $\Delta T\ =\ \binom{n}{m}\left|t(\tau,Q_1)-t(\tau,Q_2)\right|.$ Given $X=(x_1,\ldots,x_m)\in[n]^m_<$ and $A=(a_1,\ldots,a_m)\in[n]^m_<$, we consider, for $\ell \in \{1,2\}$, the quantity $Y_{\ell}(i)=Q_{\ell}(x_i,a_{\tau(i)})-Q_{\ell}(x_i,a_{\tau(i)}-1)$. By the definition of subpermutation density in a weighted permutation, we have
\[
  \Delta T\ =\ \Big|\sum_{X\in[n]^m_<}\sum_{A\in[n+1]^m_<}\Big(\prod_{i=1}^m Y_1(i)-\prod_{i=1}^m Y_2(i)\Big)\Big|
\]
\[
  =\ \Big|\sum_{X\in[n]^m_<}\sum_{A\in[n+1]^m_<}\sum_{i=1}^m\Big(Y_1(i)-Y_2(i)\Big)
  \prod_{u=1}^{i-1} Y_1(u)\prod_{v=i+1}^m Y_2(v)\Big|.
\]
Let $X_{(i)}=(x_1,\ldots,x_{i-1},x_{i+1},\ldots,x_m)\in[n]^{m-1}_<$ be the vector obtained from $X$ by the removal of the entry $x_i$, and let $A_{(i)}\in [n+1]^{m-1}_<$ be the corresponding vector for the removal of the entry $a_{\tau(i)}$ from $A$. For use in summations, we set $x_0=0$, $x_{m+1}=n+1$, $a_0=0$ and $a_{m+1}=n+1$. Using this, the above equation becomes
\[
  \Delta T\ =\ \Big|\sum_{i=1}^m\sum_{X_{(i)}}\sum_{x_i=x_{i-1}+1}^{x_{i+1}-1}
  \sum_{A_{(i)}}\sum_{a_{\tau(i)}=a_{\tau(i)-1}+1}^{a_{\tau(i)+1}-1}\Big(Y_1(i)-Y_2(i)\Big)
  \prod_{u=1}^{i-1} Y_1(u)\prod_{v=i+1}^m Y_2(v)\Big|
\]
\[
  \leq\sum_{i=1}^m\sum_{X_{(i)}}\sum_{A_{(i)}}
  \Big|\sum_{x_i=x_{i-1}+1}^{x_{i+1}-1}
  \sum_{a_{\tau(i)}=a_{\tau(i)-1}+1}^{a_{\tau(i)+1}-1}\Big(Y_1(i)-Y_2(i)\Big)\Big| \prod_{u=1}^{i-1} Y_1(u)\prod_{v=i+1}^m Y_2(v).
\]
Observe that
\[
  \sum_{a_{\tau(i)}=a_{\tau(i)-1}+1}^{a_{\tau(i)+1}-1}\Big(Y_1(i)-Y_2(i)\Big)\ =\
  \Big(Q_1(x_i,a_{\tau(i)+1}-1)-Q_1(x_i,a_{\tau(i)-1})\Big)\
\]
\[\ \ \ \ \ \ \ \ \ \ \ \ \ \ \ \ \ \ \ \ \ \ \ \ \ \ \ \ \ \ \ \ \ \ \ \ \ \ \  
  - \Big(Q_2(x_i,a_{\tau(i)+1}-1)-Q_2(x_i,a_{\tau(i)-1})\Big),
\]

Definition~\ref{def.dist.pond} now leads to
\[
  \Delta T\leq\sum_{i=1}^m\sum_{X_{(i)}\in[n]^{m-1}_<}\sum_{A_{(i)}\in[n+1]^{m-1}_<}
 d_{\square}(Q_1,Q_2)\, n \prod_{u=1}^{i-1} Y_1(u)\prod_{v=i+1}^m Y_2(v)
\]
\[
  =\ d_{\square}(Q_1,Q_2)\, n\sum_{i=1}^m\sum_{X_{(i)}\in[n]^{m-1}_<}\sum_{A_{(i)}\in[n+1]^{m-1}_<}
  \prod_{u=1}^{i-1} Y_1(u)\prod_{v=i+1}^m Y_2(v).
\]
Also note that
\[
  \sum_{A_{(i)}\in[n+1]^{m-1}_<}\prod_{u=1}^{i-1} Y_1(u)\prod_{v=i+1}^m Y_2(v)\ \leq\ 1,
\]
since we may eliminate the variables of $A_{(i)}$ one by one, as follows. To simplify notation, we suppose here that  $i\not\in[\tau(m)-1,\tau(m)+1]$. Then
\[
  \sum_{A_{(i)}}\prod_{u=1}^{i-1} Y_1(u)\prod_{v=i+1}^m Y_2(v)\
%\]
%\[
  =\sum_{A'_{(i)}\in[n+1]^{m-2}_<}\prod_{u=1}^{i-1} Y_1(u)
  \prod_{v=i+1}^{m-1} Y_2(v)\sum_{a_{\tau(m)}=a_{\tau(m)-1}+1}^{a_{\tau(m)+1}-1}Y_2(m),
\]
where $A'_{(i)}\in[n+1]^{m-2}_<$ is obtained from $A_{(i)}\in[n+1]^{m-1}_<$ by removing the coordinate $a_{\tau(m)}$. Now,
\[
  \sum_{a_{\tau(m)}=a_{\tau(m)-1}+1}^{a_{\tau(m)+1}-1}Y_2(m)\
  =\ \Big(Q_2(x_m,a_{\tau(m)+1}-1)-Q_2(x_m,a_{\tau(m)-1})\Big)\ \leq\ 1.
\]

As a consequence
\[
  \Delta T\leq\ d_{\square}(Q_1,Q_2)\cdot n\cdot m\cdot \binom{n}{m-1}.
\]
Dividing this equation by $\binom{n}{m}$ and using the inequality $n\geq 2m$,
we have
\[
  \Big|t(\tau,Q_1)-t(\tau,Q_2)\Big|\ \leq\ d_{\square}(Q_1,Q_2)\cdot n\cdot m\cdot\frac{m}{n-m+1}\
  \leq\ 2m^2\cdot d_{\square}(Q_1,Q_2),
\]
concluding our proof.
\end{proof}

\section{Proofs of Claims~\ref{eq13b} and~\ref{claim_UB2}}\label{sec.claim.proofs}

For completeness, the proofs of Claims~\ref{eq13b} and~\ref{claim_UB2} are provided in this section. 

\begin{proof}[Proof of Claim~\ref{eq13b}] The fact that, by Lemma~\ref{lemaLS.5.1}, $P_{m,j}$ is a weak $(1/j)$-regular $k_j$-partition of $\sigma_m'$ immediately implies
\begin{equation}\label{eq12}
d_{\square}(Q_{\sigma'_m},R_{m,j})\leq\frac{1}{j}.
\end{equation}
We claim that
\begin{equation}\label{tapsi3}
d_{\square}(R_{m,j},S_{m,j})\leq 2d_{\square}(Q_{m,j},Q_j)+\frac{2}{k_j}+\frac{2}{|\sigma_m'|},
\end{equation}
so that part (ii) of Lemma~\ref{lemaLS.5.1} leads to
\begin{equation}\label{eq13}
d_{\square}(R_{m,j},S_{m,j}) \leq 2\left(\frac{1}{j}+\frac{1}{k_j}+\frac{1}{|\sigma_m'|}\right).
\end{equation}
Claim~\ref{eq13b} now follows from the triangle inequality, together with equations~(\ref{eq12}) and~(\ref{eq13}).

To establish~(\ref{tapsi3}), recall that, given an equitable partition  $P_{m,j}=\bigcup_\ell C_{m,j,\ell}$ of $\sigma_m'$, the matrix $R_{m,j}$ is obtained from $Q_{m,j}$ by replacing each single entry $(r,s)$ by a block of size $|C_{m,j,r}||C_{m,j,s}|$, all of whose entries assume the same value. This also holds for $S_{m,j}$ and $Q_j$. Now, to express $d_{\square}(R_{m,j},S_{m,j})$ as in Definition~\ref{def.dist.pond}, let $S=[s_0,s_1] \in I[|\sigma_m'|]$ and let $a<b \leq |\sigma_m'|+1$. Let $\ell_0, \ell_1, \ell_a, \ell_b \in [k_j]$ be such that $s_0 \in C_{m,j,\ell_0}$, $s_1 \in C_{m,j,\ell_1}$, $a \in C_{m,j,\ell_a}$ and $b \in C_{m,j,\ell_b}$ respectively. Fix $s_0'=\sum_{\ell=1}^{\ell_0}|C_{m,j,\ell}|$ and $s_1'=\sum_{\ell=1}^{\ell_1-1}|C_{m,j,\ell}|$. Observe that the summation  
$$\sum_{x\in S}\Big((R_{m,j}(x,b)-R_{m,j}(x,a))-(S_{m,j}(x,b)-S_{m,j}(x,a))\Big)$$ may be split into three intervals, namely $[s_0,s_0']$, $[s_0'+1,s_1']$ and $[s_1'+1,s_1]$. Because our partition is equitable, the first and third interval contain at most $\lceil |\sigma_m'|/k_j \rceil$ terms, all of which have value in the interval $[-1,1]$. As a consequence,
\[
\Big|\sum_{x\in S \setminus [s_0',s_1']}\Big((R_{m,j}(x,b)-R_{m,j}(x,a))-(S_{m,j}(x,b)-S_{m,j}(x,a))\Big) \Big| < 2\left\lceil\frac{\sigma_m'}{k_j}\right\rceil.
\]
On the other hand, it is easy to see that
\begin{equation*}
\begin{split}
\Big|\sum_{x=s_0'}^{s_1'}&\Big((R_{m,j}(x,b)-R_{m,j}(x,a))-(S_{m,j}(x,b)-S_{m,j}(x,a))\Big) \Big| \\
&= \Big|\sum_{y=\ell_0}^{\ell_1-1}|C_y|\Big((Q_{m,j}(y,\ell_b)-Q_{m,j}(y,\ell_a))-(Q_{j}(y,\ell_b)-Q_{j}(y,\ell_a))\Big) \Big| \\
& \leq \left\lceil\frac{|\sigma_m'|}{k_j}\right\rceil \Big|\sum_{y=\ell_0}^{\ell_1-1}\Big((Q_{m,j}(y,\ell_b)-Q_{m,j}(y,\ell_a))-(Q_{j}(y,\ell_b)-Q_{j}(y,\ell_a))\Big) \Big|\\
& \leq \left\lceil\frac{|\sigma_m'|}{k_j}\right\rceil k_j d_{\square}(Q_{m,j},Q_j).
\end{split}
\end{equation*}
We conclude that
$$d_{\square}(R_{m,j},S_{m,j}) \leq \frac{1}{|\sigma_m'|} \cdot  2\left\lceil\frac{|\sigma_m'|}{k_j}\right\rceil + \frac{k_j}{|\sigma_m'|}\left\lceil\frac{|\sigma_m'|}{k_j}\right\rceil  d_{\square}(Q_{m,j},Q_j).$$
Our claim now follows from the fact that $\left\lceil\frac{|\sigma_m'|}{k_j}\right\rceil \leq \frac{|\sigma_m'|}{k_j}+1$ and $\frac{k_j}{|\sigma_m'|} \leq 1$. \end{proof}

\begin{proof}[Proof of Claim~\ref{claim_UB2}] The proof here resembles that of the previous claim, therefore we only provide a sketch. By Lemma~\ref{lema_auxiliar}, the quantities
$|t(\tau,Z_{m,j})-t(\tau,S_{m,j})|$ and $|t(\tau,Z_{j})-t(\tau,Q_{j})|$ can be made sufficiently small for any fixed permutation $\tau$ if we choose $m$ and $j$ is sufficiently large. We use Definition~\ref{def.dens.perm2} and relate the terms in the expression
\begin{equation}\label{tapsi4}
\sum_{X\in[|\sigma_m'|]^{|\tau|}\cdot}\sum_{A\in[|\sigma_m'|+1]^{|\tau|}\cdot}
\prod_{i=1}^{|\tau|}\Big(S_{m,j}(x_i,a_{\tau(i)})-S_{m,j}(x_i,a_{\tau(i)}-1)\Big)
\end{equation}
with the terms in
\begin{equation}\label{tapsi5}
\sum_{X\in[k_j]^{|\tau|}\cdot}\sum_{A\in[k_j+1]^{|\tau|}\cdot}
\prod_{i=1}^{|\tau|}\Big(Q_{j}(x_i,a_{\tau(i)})-Q_j(x_i,a_{\tau(i)}-1)\Big).
\end{equation}
It is easy to see that, for a fixed $x_j$, every term in the inner sum of~(\ref{tapsi5}) corresponds to $\displaystyle{|C_{m,j,l}|^{|\tau|}}$ identical terms in~(\ref{tapsi4}), where $P_{m,j}=(C_{m,j,l})_{l=1}^{k_j}$ is the equitable partition of $\sigma_m'$ originating $S_{m,j}$. This quantity may be approximated by $\displaystyle{\left(|\sigma_{m}'|/k_j\right)^{|\tau|}}$. Our results follows immediately from fact that
$$\binom{|\sigma_m'|}{|\tau|}^{-1}\binom{k_j}{|\tau|}\left(\frac{|\sigma_{m}'|}{k_j}\right)^{|\tau|} \to 1$$
as $m$ tends to infinity. \end{proof}

\end{document}